\documentclass{amsart}
\usepackage{color}
\usepackage[utf8]{inputenc}
\usepackage{amscd}
\usepackage{pstricks}
\usepackage{amsmath}
\usepackage[english]{babel}
\usepackage{amsfonts}
\usepackage{graphicx}
\usepackage{vmargin}
\usepackage{theoremref}
\usepackage[hidelinks]{hyperref}
\usepackage{amssymb}
\usepackage{tikz}
\usepackage{relsize}
\usepackage[toc,page]{appendix}
\usepackage{commath}
\usepackage{mathrsfs}
\usetikzlibrary{matrix}
\usepackage{cancel}
\usepackage{lineno}
\usepackage{comment}
\usepackage[colorinlistoftodos,prependcaption,textsize=tiny]{todonotes}

\newcommand{\rank}{{\mathrm{rank}}}

\newcommand{\Id}{\mathrm{Id}}
\newcommand{\inv}{\mathrm{inv}}
\newcommand{\ninv}{\mathfrak{inv}}
\newcommand{\e}{\mathfrak{e}}

\newcommand{\stab}{\mathrm{stab}}
\newcommand{\im}{\mathrm{im}}
\newcommand{\rimes}{\!\times\! }
\renewcommand{\P}{\mathrm{P}}

\newtheorem{theorem}{Theorem}[section]
\newtheorem{lemma}[theorem]{Lemma}
\newtheorem{conjecture}[theorem]{Conjecture}
\newtheorem{definition}[theorem]{Definition}
\newtheorem{proposition}[theorem]{Proposition}
\newtheorem{corollary}[theorem]{Corollary}

\newtheorem{example}[theorem]{Example}

\newtheorem{remark}[theorem]{Remark}

\usepackage{chngcntr}
\counterwithin{figure}{section}

\numberwithin{equation}{section}

\title[]{Generalized nonautonomous dynamics through groupoid morphisms}%Groupoid morphisms as an algebraic structure for nonautonomous dynamics/}
\author[]{N\'estor Jara \& Emir Molina}

\address{Universidad de Chile, Departamento de Matem\'aticas. Casilla 653, Santiago, Chile}

\email{nestor.jara@ug.uchile.cl; emir.molina@ug.uchile.cl}
\thanks{This research has been partially supported by ANID, Beca de Doctorado Nacional 21220105.}
\date{\today}

\begin{document}

\maketitle

\begin{abstract}
We extend the notions of nonautonomous dynamics to arbitrary groups, through groupoid morphisms. This also presents a generalization of classic dynamical systems and group actions. We introduce the structure of cotranslations, as a specific kind of groupoid morphism, and establish a correspondence between cotranslations and skew-products. We give applications of cotranslations to nonautonomous equations, both in differences and differential. Our results delve into the differentiability of cotranslations, along with dimension invariance and diagonalization, utilizing a generalized notion of kinematic similarity. 
\end{abstract}

\section{Introduction}
The theory of dynamical systems is currently one of the most active areas of research in mathematics, with autonomous dynamics being the primary focus. In contrast, the field of nonautonomous dynamics has made substantial advancements in the past two decades, particularly in the study of time-dependent differential and difference equations (e.g., \cite{Potzsche2}, \cite{BV}, \cite{Cheban}, \cite{Johnson}, \cite{Kloeden}, \cite{Potzsche1}, \cite{Rasmussen}). There is a significant disparity between the previously described concepts, while autonomous dynamics depend only on the time elapsed from the initial time $t_{0}$, the nonautonomous systems are also dependent on the initial time $t_{0}$ itself, which has several consequences to characterize limiting objects. As it was pointed out in \cite[Ch.2]{Kloeden}, the dynamics arising from nonautonomous differentiable system
\begin{equation}
\label{sistema2}
\dot{x}=f(t,x)
\end{equation}
can be formally described by two approaches where the above mentioned $t_{0}$ plays a key role: the skew-product semiflows and the process formalism, also known as the two parameter-$(t,t_0)$-semigroups.

We emphasize that an autonomous differential equation does indeed define a continuous group action, {\it i.e.} a classic dynamical system, but as soon as we consider a nonautonomous differential equation as (\ref{sistema2}), these structures no longer describe properly the dynamics given by the flow of solutions.

\subsection{Preliminaries and setting}Consider a category $\mathscr{C}$ (topological spaces, Banach spaces, rooted tree graphs, among others) and an element $X$ on said category. We also consider
\begin{itemize}
    \item $\mathbb{B}_X$ the collection of all the morphisms of the category $\mathscr{C}$ of $X$ on itself,
    \item $\mathbb{A}_X$ the collection of all invertible elements of $\mathbb{B}_X$ (whose inverse is also in $\mathbb{B}_X$). We know that $\mathbb{A}_X$ is a group.
\end{itemize}

Depending on the category $\mathscr{C}$, both sets $\mathbb{B}_X$ and $\mathbb{A}_X$ may have more structure, but for now we keep it general.

\smallskip

A dynamical system, regardless of the category $\mathscr{C}$, involves a group $G$ (possibly with additional structure) and a group morphism $\gamma: G \to \mathbb{A}_X$. In the topological case, where $X$ is a topological space and $G$ is a topological group, we equip $\mathbb{A}_X$ with the compact-open topology \cite[Definition I, p. 301]{Bourbaki}. In this case the group morphism $\gamma$ is required to be continuous.

\smallskip

Equivalently, we may define that a left (topological) dynamical system is a triple $(X,G,\alpha)$, where $X$ and $G$ satisfy the same conditions as before and $\alpha$ is a continuous left action of $G$ on $X$, {\it i.e.} a map $\alpha:G\rimes X\to X$ verifying that each $\alpha(g,\cdot)$ is a homeomorphism of $X$ on itself and
$$\alpha(g,\alpha(h,x))=\alpha(gh,x), \qquad \forall\,g,h\in G,\, x\in X.$$

It is easy to see that by setting $\hat{\alpha}:G\to \mathbb{A}_X$ by $\left[\hat{\alpha}(g)\right](x)=\alpha(g,x)$ we obtain a continuous group morphism. On the other hand, we can define right actions as maps $\beta:X\times G\to X$ verifying that each $\beta(\cdot,g)$ is a homeomorphism of $X$ on itself and
$$\beta(\beta(x,g),h)=\beta(x,gh), \qquad \forall\,g,h\in G,\, x\in X.$$

Once more, by setting $\hat{\beta}:G\to \mathbb{A}_X$ by $\left[\hat{\beta}(g)\right](x)=\beta(x,g^{-1})$ we obtain a continuous group morphism. Therefore, both left and right actions describe dynamical systems and the use of one over the other is just a matter of convenience in notation. This correspondence between left or right actions and their algebraic counterpart on group morphisms is a well known fact at the basis of classic dynamics. Nevertheless, nonautonomous dynamics \cite{Kloeden} do not have a widely discussed analogous notion of {\it action}, even less what could be its algebraic counterpart; therefore, the main objective of this work is to shed light on how to fill this lack.  Through this, we will also generalize nonautonomous dynamics to arbitrary groups.

\subsection{Novelty and Structure} The principal contribution of this article is to introduce an algebraic generalization of dynamical systems. This generalization is applicable to describe flows arising from nonautonomous equations, encompassing both differential and difference equations. Notably, it extends the notion of nonautonomous dynamics beyond the realms of specific number systems such as $\mathbb{R}$ or $\mathbb{Z}$, bringing it into the realm of general groups.

\smallskip

The paper is organized as follows. In the second section we study
the structure of {\it skew-product}, objects which, to the best of our knowledge, emerge for the first time on 1950 by H. Anzai \cite{Anzai}, whom uses them to describe a certain ergodic dynamic, and then later on 1965 were connected to differential equations thanks to the work of R. K. Miller \cite{Miller}. This concept refers to a generalization of dynamical systems given by left actions, but instead of considering one action, it uses a family of action-like functions with a certain compatibility relation (we give more details later). The skew-products we present here are defined for general groups, not just the usual $\mathbb{R}$ or $\mathbb{Z}$.

\smallskip

On the third section we present groupoids and groupoid morphisms. We define a specific type of these groupoid morphisms called cotranslations, which will give an algebraic structure of the dynamics that can be represented by skew-products. We also state some results regarding the relation of these groupoid morphisms to  discrete nonautonomous dynamics. We present several examples of this construction for different groups.

\smallskip
On the fourth section we study differentiable groupoid morphisms, give some basic properties and give an application to the problem of existence of solutions to linear nonautonomous differential equations on Banach spaces.

\smallskip
On the final section we study a partial notion of the groupoid morphisms we presented earlier, mainly on the Euclidean space, and give some of their algebraic and topological properties, obtaining a generalized notion of kinematic similarity.

\section{Skew-product dynamical systems}
In this section we present skew-products dynamical systems. Although we present them in the topological case, a similar structure can be defined and studied for objects on different categories. As described by R. J Sacker \cite{Sacker2}, a  (topological) skew-product is constituted from a topological space $X$ and a topological group $G$, both will be fixed unless stated otherwise. For simplicity, we always consider locally compact Hausdorff topologies on $X$ and $G$. In this context, $\mathbb{A}_X$ is the group of all the homeomorphisms of $X$ on itself, given the compact-open topology. The main characteristic of this generalization of dynamical systems is that instead of a unique group action, we consider a family of action-like maps. It is worth noting that in \cite{Elaydi,Sacker,Sacker2} the groups are $\mathbb{R}$ or $\mathbb{Z}$, while here we present the construction in general.

\smallskip

For locally compact Hausdorff spaces $A$ and $B$, we set $\mathcal{C}\big(A;B\big)$ the space of continuous functions from $A$ to $B$, given the compact-open topology. Let us consider the evaluation maps:
\begin{itemize}
    \item $\mathfrak{e}_G:\mathcal{C}\big(G; X\big)\times G\to X$, $(\varphi,g)\mapsto \varphi(g)$,
    \item $\mathfrak{e}_X:\mathcal{C}\big(X; X\big)\times G\to X$, $(\varphi,x)\mapsto \varphi(x)$,
\end{itemize}
as well as the partial evaluation maps:
\begin{itemize}
    \item $\widetilde{\mathfrak{e}}_G:\mathcal{C}\big(G\!\times\!X; X\big)\times G\to \mathcal{C}\big(X; X\big)$, $(\psi,g)\mapsto \psi(g,\cdot)$,
    \item $\widetilde{\mathfrak{e}}_X:\mathcal{C}\big(G\!\times\!X; X\big)\times X\to \mathcal{C}\big(G; X\big)$, $(\psi,x)\mapsto \psi(\cdot,x),$
\end{itemize}
which are all continuous \cite[Corollary I and II, p. 303]{Bourbaki}. For a subspace $Y\subset \mathcal{C}\big(G\!\times\!X; X\big)$, define 
$$Y_X:=\widetilde{e}_X(Y\rimes X)\subset \mathcal{C}\big(G; X\big)\quad \text{and}\quad Y_G:=\widetilde{e}_G(Y\rimes G)\subset \mathcal{C}\big(X; X\big).$$

\begin{definition}
    We say that a space of functions $Y\subset \mathcal{C}\big(G\!\times\!X; X\big)$ is \textbf{admissible}, if:
    \begin{itemize}
        \item [i)] $Y_G\subset\mathbb{A}_X$
        \item [ii)] $\psi(e,x)=x$ for every $\psi\in Y$ and $x\in X$, where $e$ is the unit of $G$.
    \end{itemize}
    
    In \cite{Sacker}, such a space $Y$ is called a \textbf{Hull}.
\end{definition}

On the other hand, if we consider the right action $\theta$ of $G$ on itself by translations, {\it i.e.}
$$\theta:G\times G\to G,\quad (g,h)\mapsto gh,$$
it lifts to a continuous left action $\Theta:G\!\times\!\mathcal{C}\big(G; X\big)\to \mathcal{C}\big(G; X\big)$ given by
$$\left[\Theta(h,\varphi)\right](g)=\varphi\left(\theta(g,h)\right)=\varphi(gh),\quad \forall\,\varphi\in \mathcal{C}\big(G; X\big),\,g,h\in G.$$
 
For an admissible $Y$, the set $Y_X$ is contained on $\mathcal{C}\big(G;X\big)$. If we consider its saturation $\widetilde{Y}_X:=\Theta\big(G\!\times\!Y_X\big)$, it is invariant under the action $\Theta$, hence $\left(\widetilde{Y}_X,G,\Theta\right)$ is a dynamical system.

\smallskip
With this system, we write the map $\widetilde{\Theta}:X\rimes Y\rimes G\rimes G\to X$ by 
\begin{equation}\label{192}
    \widetilde{\Theta}(x,\psi,g,h)=\left[\Theta\left(h,\widetilde{\e}_X(\psi,x)\right)\right](g)=\psi(gh,x).
\end{equation}

Now consider a dynamical system $(Y,G,\sigma)$, given by a continuous left action $\sigma:G\times Y\to Y$. The {\it skew-flow} associated to it is the map $\pi:X\times Y\times G\to X\times Y$ given by 
$$\pi(x,\psi,h)=\left(\psi(h,x),\sigma(h,\psi)\right),\quad\forall\,\psi\in Y,\,h\in G,\,x\in X,$$
which is how skew-products are depicted on \cite{Sacker}. With this system, we define the map $\widetilde{\pi}:X\rimes Y\rimes G\rimes G\to X$ by
\begin{equation}\label{193}
\widetilde{\pi}(x,\psi,g,h)=\left[\sigma\big(h,\psi\big)\right]\left(g,\psi\big(h,x\big)\right),\quad\forall\,\psi\in Y,\,g,h\in G,\,x\in X.
\end{equation}

The key to skew-products is that the systems given by $\Theta$ and $\sigma$ must be compatible. This means that the maps $\widetilde{\Theta}$ and $\widetilde{\pi}$, on (\ref{192}) and (\ref{193}) respectively, must coincide, {\it i.e.}
$$\left[\sigma(h,\psi)\right]\left(g,\psi(h,x)\right)=\psi(gh,x) ,\quad\forall\,\psi\in Y,\,g,h\in G,\,x\in X.$$

In this last statement, it is clear that all the information regarding the skew-product is contained on the properties of $\sigma$. We formalize this discussion on the following definition.
\begin{definition}\label{300}
    A \textbf{skew-product dynamical system} is a quadruple $(X,G,Y,\sigma)$, where:
    \begin{itemize}
    \item [i)] $Y\subset \mathcal{C}\big(G\!\times\!X;X\big)$ is admissible,
        \item [ii)] $(Y,G,\sigma)$ is a dynamical system, where $\sigma:G\!\times\! Y\to Y$ is a continuous left action,
        \item [iii)] $\left[\sigma(h,\psi)\right]\left(g,\psi(h,x)\right)=\psi(gh,x)$ for every $\psi\in Y,\,g,h\in G$ and $\,x\in X$.
    \end{itemize}
\end{definition}

\begin{remark}%\label{137}
   {\rm  It is clear that if $(X,G,\alpha)$ is a dynamical system, where $\alpha$ is a continuous left action, then $(X,G,\{\alpha\},\mathfrak{q})$ is a skew-product, where $\mathfrak{q}:G\rimes\{\alpha\}\to \{\alpha\}$ is the trivial action. In other words, a dynamical system is a skew-product where the hull contains only one function.}
\end{remark}

Let us illustrate this construction on an example.

\begin{example}%\label{130}
{\rm
Consider the topological space $\mathbb{R}^d$ and its group of homeomorphisms $\mathbb{A}_X$. Set $G=\mathbb{Z}$. Consider the following nonautonomous difference equation 
    \begin{equation}\label{121}
        x(n+1)=F\left(n,x(n)\right),
    \end{equation} 
    where for every $n\in \mathbb{Z}$ we have $F(n,\cdot)\in \mathbb{A}_X$. Let $n\mapsto x(n,m,\xi)$ be the unique solution of \eqref{121} that verifies $x(m,m,\xi)=\xi$. Define the collection $Y=\left\{\psi_{m}\,:\,m\in \mathbb{Z}\right\}$, where $\psi_m:\mathbb{Z}\rimes \mathbb{R}^d\to \mathbb{R}^d$ is given by
    $$\psi_{m}(n,\xi)=x\big(n+m,m,\xi\big).$$ 
    
    Set now the action $\sigma:Y\!\times\! \mathbb{Z}\to Y$ by
    $$\sigma(n,\psi_{m})=\psi_{n+m},\quad\forall\,n,m\in \mathbb{Z},$$ 
    or, evaluating
    \begin{align*}
    \left[\sigma\left(\psi_{m},n\right)\right](p,\xi)=\psi_{n+m}(p,\xi)=x\big(p+n+m,n+m,\xi\big),\quad\forall\,p,n,m\in \mathbb{Z},\,\xi\in \mathbb{R}^d,
    \end{align*}
    and its associated skew-flow $\pi:\mathbb{R}^d\!\times\!Y\!\times\!\mathbb{Z}\to \mathbb{R}^d\!\times\!Y$ given by $\pi\left(\xi,\psi_{m},n\right)=\left(\psi_{m}(n,\xi),\psi_{n+m}\right)$. Evaluating for $p,n,m\in \mathbb{Z}$ and $\xi\in \mathbb{R}^d$ we have
\begin{align*}
        \widetilde{\pi}(\xi,\psi_m,p,n)&=\psi_{n+m}\left(p,\psi_{m}\big(n,\xi\big)\right)\\
        &=x\left(p+n+m,n+m,x\big(n+m,m,\xi\big)\right).
    \end{align*}

On the other hand, in this context
\begin{align*}
    \widetilde{\Theta}(\xi,\psi_m,p,n)=\psi_{m}(p+n,\xi)=x\big(p+n+m,m,\xi\big),\quad\forall\,p,n,m\in \mathbb{Z},\,\xi\in \mathbb{R}^d.
\end{align*}

Now, by uniqueness of solutions we have
$$x\left(p+n+m,n+m,x\big(n+m,m,\xi\big)\right)=x\big(p+n+m,m,\xi\big), \quad\forall\,p,n,m\in \mathbb{Z},\,\xi\in \mathbb{R}^d,$$
thus $\widetilde{\Theta}$ and $\widetilde{\pi}$ coincide, which implies that $\left(\mathbb{R}^d,\mathbb{Z},\{\psi_m:m\in \mathbb{Z}\},\sigma\right)$ is indeed a skew-product.   }    
\end{example}

In \cite{Elaydi}, S. Elaydi and R. J. Sacker explored further applications of skew-products to the theory of nonautonomous difference equations, as the search of asymptotically stable solutions for Beverton-Holt equations \cite{Cushing}. Moreover, in \cite{Sacker}, the authors leverage this structure to develop the exponential dichotomy spectrum for nonautonomous linear differential equations when the hull is compact.

\section{Groupoid morphisms and cotranslations}

On this section, we introduce the structure of cotranslations, a specific kind of groupoid morphisms. Cotranslations serve as an algebraic counterpart to skew-products, providing a generalization of group morphisms. As group morphisms are instrumental in describing dynamical systems, this generalization further extends the applicability of such algebraic concepts.

\begin{definition} \cite[Definition 1.2]{Williams}
    We say that a set $\Xi$, doted of a subset $\Xi^{(2)}\subset \Xi\rimes \Xi$ (called the collection of composible pairs) and two maps $\bullet:\Xi^{(2)}\to \Xi$ given by $(\eta,\xi)\mapsto\eta\bullet\xi$ (called composition law), and $\ninv:\Xi\to \Xi$, is a \textbf{groupoid} if the following conditions are verified
    \begin{itemize}
        \item [i)] (associativity) If $(\eta,\xi),\,(\xi, \zeta)\in \Xi^{(2)}$, then  $(\eta\bullet\xi,\zeta),\,(\eta,\xi\bullet\zeta)\in \Xi^{(2)}$ and $(\eta\bullet\xi)\bullet\zeta=\eta\bullet(\xi\bullet\zeta)$,
        \item [ii)] (involution) $\ninv(\ninv(\eta))=\eta$ for every $\eta\in \Xi$,
        \item [iii)] (identity) for every $\eta\in \Xi$, we have $\left(\eta,\ninv(\eta)\right)\in\Xi^{(2)}$ and $(\eta,\xi)\in\Xi^{(2)}$ implies that $\ninv(\eta)\bullet(\eta\bullet\xi)=\xi$ and $(\eta\bullet\xi)\bullet\ninv(\xi)=\eta$.
    \end{itemize}
    
    Furthermore, we define $\Xi^{(0)}:=\left\{\eta\in \Xi:\eta=\ninv(\eta)=\eta\bullet\eta\right\}$ and call it the \textbf{units space} of the groupoid.
\end{definition}

A trivial example of a groupoid is a group $G$. In this case $G^{(2)}=G^2$ and $G^{(0)}=\{e\}$. 

\begin{example}%\label{177}
{\rm
If a group $G$ acts by the left on a set $M$, the product $M\rimes G$ has groupoid structure. Indeed, setting
$$\left(M\rimes G\right)^{(2)}:=\left\{\left((x,g),(y,h)\right)\in \left(G\rimes M\right)^{2}:x=h\cdot y\right\},$$
and
$$\bullet\left((x,g),(y,h)\right)=(y,gh),\quad\ninv(x,g)=( g\cdot x,g^{-1}),$$
the groupoid axioms are easily followed. In this case, the units space corresponds to the collection of points of the form $(x,e)$, where $e$ is the group unit.% Analogously we can define a groupoid for right actions.

\smallskip
This example is particularly useful when $G$ acts on itself by left translations, in which case we call $G\rimes G$ the \textbf{left translations groupoid} for $G$. 
}
\end{example}

\begin{definition}  \cite[Definition 1.8]{Williams}
    If $\Xi$ and $\Upsilon$ are groupoids, a \textbf{groupoid morphism} is a map $\vartheta:\Xi\to \Upsilon$ verifying that $(\eta,\xi)\in \Xi^{(2)}$ implies $\left(\vartheta(\eta),\vartheta(\xi)\right)\in \Upsilon^{(2)}$ and $\vartheta(\eta\bullet\xi)=\vartheta(\eta)\bullet\vartheta(\xi)$.
\end{definition}

Two basic facts about groupoid morphisms are that they preserve involutions and units spaces. In other words, for a groupoid morphism $\vartheta:\Xi\to\Upsilon$ we have $\ninv(\vartheta(\xi))=\vartheta(\ninv(\xi)$ for all $\xi\in \Xi$ and $\vartheta(\eta)\in \Upsilon^{(0)}$ for all $\eta\in \Xi^{(0)}$.

\begin{example}\label{159}
{\rm    Consider an element $X$ on some category $\mathscr{C}$. Set a group $G$,  $\gamma:G\to \mathbb{A}_X$ a group morphism and give $G\rimes G$ the left translations groupoid structure. By setting $Z:G\rimes G\to \mathbb{A}_X$ by $Z(g,h)=\gamma(h)$ we obtain a groupoid morphism. Indeed:
$$Z(g,kh)=\gamma(kh)=\gamma(k)\gamma(h)=Z(hg,k)Z(g,h).$$
}
\end{example}

This example shows that every dynamical system is in particular given by a groupoid morphism, since a dynamical system is always given by a group morphism $\gamma:G\to \mathbb{A}_X$. Nevertheless, the next example shows that groupoid morphisms describe a more general kind of dynamics.

\begin{example}\label{191}
    {\rm Consider the left translations groupoid structure on $\mathbb{R}\rimes\mathbb{R}$. Take $X=\mathbb{R}^d$ as a topological space, thus $\mathbb{A}_X$ is the group of all of its homeomorphisms. Consider a nonautonomous real differential equation $\dot{x}=F(t,x)$ such that for every $\xi\in \mathbb{R}^d$ and every $r\in \mathbb{R}$, there is a unique and globally defined solution $x_{r,\xi}:\mathbb{R}\to\mathbb{R}^d$ such that  $x_{r,\xi}(r)=\xi$. Then the map $Z:\mathbb{R}\rimes\mathbb{R}\to \mathbb{A}_X$ given by $\left[Z(r,t)\right](\xi)=x_{r,\xi}(t+r)$ is a groupoid morphism. Indeed:
    $$\left[Z(r,t+s)\right](\xi)=x_{r,\xi}(t+s+r)=x_{s+r,x_{r,\xi}(s+r)}(t+s+r)=\left[Z(s+r,t)\circ Z(r,s)\right](\xi),$$
where the second equality is a well known fact deduced from the 
uniqueness of solutions. Moreover, giving $\mathbb{A}_X$ the compact-open topology and the groupoid $\mathbb{R}\rimes \mathbb{R}$ the product topology, $Z$ turns out to be continuous.
    
    }
\end{example}

Given the dynamical relevance shown by the previous examples, we give these morphisms a distinctive name.

\begin{definition}
    Consider a group $G$ and an object $X$ on a category. A \textbf{cotranslation} is a groupoid morphism $Z:G\rimes G\to \mathbb{A}_X$, where $G\rimes G$ has the left translations groupoid structure.
    \end{definition}

Now we present the main theorem of this work. Although we state it for the topological case, it can be generalized for objects and skew-products on different categories.

\begin{theorem}\label{325}
    Let $X$ be a locally compact Hausdorff topological space and $G$ a locally compact Hausdorff topological group. Give all function spaces, including $\mathbb{A}_X$, the compact-open topology.
    
    \smallskip
    There is a bijective correspondence between skew-product dynamical systems $(X,G,Y,\sigma)$, where the action $\sigma$ is transitive, and continuous cotranslations $Z:G\rimes G\to \mathbb{A}_X$.
\end{theorem}

\begin{proof}
 Let $(X,G,Y,\sigma)$ be a skew-product where $\sigma$ is transitive. The admissibility condition implies $Y_G\subset \mathbb{A}_X$. On the other hand, as $\sigma$ is transitive, we can identify $Y$ with the quotient group,  $G/\stab(\sigma)$, where 
    $$\stab(\sigma)=\left\{g\in G: \sigma(g,y)=y,\,\forall\, y\in Y\right\}.$$

    Hence, we write $Y=\left\{\psi_{\overline{g}}:g\in G\right\}$, where $\overline{g}$ denotes the class of $g$ on the quotient $G/\stab(\sigma)$. Moreover, the action $\sigma:G\rimes Y\to Y$ is rewritten as $\hat{\sigma}:G\rimes Y\to Y$
given by $\hat{\sigma}(\psi_{h,\overline{g}})=\psi_{\overline{hg}}$, that is, it is identified with the action of left translations of $G$ on $G/\stab(\sigma)$.

    \smallskip
Now, we can define the continuous map $Z:G\rimes G\to \mathbb{A}_X$ given by$$Z(g,h)=\widetilde{\e}_G(\psi_{\overline{g}},h)=\psi_{\overline{g}}(h,\cdot),$$
 and we have
    \begin{eqnarray*}
        Z(g,kh)=Z(hg,k)\circ Z(g,h)&\Leftrightarrow&\left[Z(g,kh)\right](x)=\left[Z(hg,k)\circ Z(g,h)\right](x),\quad\forall\,x\in X\\
        &\Leftrightarrow&\left[\widetilde{\e}_G(\psi_{\overline{g}},kh)\right](x)=\widetilde{\e}_G(\psi_{\overline{hg}},k)\left[\left[ \widetilde{\e}_G(\psi_{\overline{g}},h)\right](x)\right],\quad\forall\,x\in X\\
        &\Leftrightarrow&\psi_{\overline{g}}(kh,x)=\psi_{\overline{hg}}\left(k,\psi_{\overline{g}}(h,x)\right),\quad\forall\,x\in X\\
        &\Leftrightarrow&\psi_{\overline{g}}(kh,x)=\left[\hat{\sigma}(h,\psi_{\overline{g}})\right]\left(k,\psi_{\overline{g}}(h,x)\right),\quad\forall\,x\in X,
    \end{eqnarray*}
and as the last condition is guaranteed by third the axiom of skew-products (Definition \ref{300}), then $Z$ is indeed a cotranslation.

\smallskip
Conversely, given a continuous cotranslation $Z:G\rimes G\to \mathbb{A}_X$, for each $g\in G$ we define $\psi_g:G\rimes X\to X$ given by  $\psi_g(h,x)=\left[Z(g,h)\right](x)$. Then, defining $Y:=\left\{\psi_g:g\in G\right\}$, it is admissible, since groupoid morphisms preserve units, and the only unit in $\mathbb{A}_X$ is the identity. On the other hand, defining the action $\sigma:G\times Y\to Y$ given by $\sigma(h,\psi_g)=\psi_{hg}$, or equivalently
    $$\left[\sigma(h,\psi_g)\right](k,x)=\left[Z(hg,k)\right](x),$$
    we obtain, by the same previous argument, that $(X,G,Y,\sigma)$ is a skew-product dynamical system, where the action $\sigma$ clearly results transitive.
\end{proof}

\begin{remark}\label{201}
{\rm
    The preceding theorem demonstrates that all generalizations of dynamical systems derived from skew-products are encompassed by cotranslations. The distinctive virtue of cotranslations lies in their ability to more clearly represent the underlying algebraic structure of these dynamics. This parallel can be drawn to the way group morphisms capture the algebraic structure of group actions.}
\end{remark}

Now we give further properties and applications for cotranslations. In the following we use the notation $Z^{\inv}(g,h)=\left[Z(g,h)\right]^{-1}$.

\begin{lemma}%\label{148}
    Set a cotranslation $Z:G\rimes G\to \mathbb{A}_X$. If $e\in G$ is the group unit, then for every $g,h\in G$ we have $$Z(g,e)=\Id\quad\text{and}\quad Z^\inv(g,h)=Z(hg,h^{-1}).$$
\end{lemma}

\begin{proof}
    The first equality follows from the fact that groupoid morphisms preserve units, and the only unit in $\mathbb{A}_X$ is the identity. The second equality follows from the fact that groupoid morphisms preserve involutions.
\end{proof}

\begin{proposition}
Let $Z:G\rimes G\to \mathbb{A}_X$ be a cotranslation. Let $\gamma:G\to \mathbb{A}_X$ be a group morphism such that $$\gamma(k)Z(g,h)=Z(g,h)\gamma(k),\qquad\forall\,g,k,h\in G,$$
then $W:G\rimes G\to \mathbb{A}_X$ given by $W(g,h)=Z(g,h)\gamma(h)$ is a cotranslation.
\end{proposition}
\begin{proof}
    It is enough to see that
    \begin{align*}
        W(g,kh)=Z(g,kh)\gamma(kh)&=Z(hg,k)Z(g,h)\gamma(k)\gamma(h)\\
    &=Z(hg,k)\gamma(k)Z(g,h)\gamma(h)\\
    &=W(hg,k)W(g,h).
    \end{align*}
\end{proof}

\begin{example}
    {\rm
    Consider $X=\mathbb{R}^d$ as a Banach space. In this context, $\mathbb{B}_X\cong\mathcal{M}_d(\mathbb{R})$ and $\mathbb{A}_X\cong GL_d(\mathbb{R})$. Consider a linear nonautonomous differential equation $\dot{x}=A(t)x(t)$, where $t\mapsto A(t)$ is locally integrable. Denote the transition matrix for this equation by $\Phi:\mathbb{R}\rimes\mathbb{R}\to \mathbb{A}_X$. We can see, analogously to Example \ref{191}, that $Z:\mathbb{R}\rimes\mathbb{R}\to \mathbb{A}_X$, given by $Z(r,t)=\Phi(r,t+r)$, defines a cotranslation.

    \smallskip
    
    Choose $\lambda\in \mathbb{R}$ and define $\gamma:\mathbb{R}\to \mathbb{A}_X$ by $\gamma(t)=e^{-\lambda t} \cdot\Id$. It is easily a group morphism which verifies the conditions of the previous proposition. Hence, $W:\mathbb{R}\times\mathbb{R}\to \mathbb{A}_X$ given by $W(r,t)=Z(r,t)\gamma(t)$ is a cotranslation. Moreover, in this case $W$ is once more the cotranslation associated to a linear differential equation, since it is obtained in the same fashion as $Z$, but regarding the shifted linear nonautonomous differential equation $\dot{x}=\left[A(t)-\lambda\cdot\Id\right]x(t)$.
    }
\end{example} 

Now we present a result that depicts how a cotranslation on a given space can be regarded as a group action, {\it i.e.} a classic dynamical system, on a different suitable space. As in Theorem \ref{325}, we express this result in terms of topological dynamics.

\begin{proposition}\label{202}
  Let $X$ be a topological space and $G$ a topological group, both endowed with Hausdorff locally compact topologies. Consider $G\times X$ as topological space equipped with the product topology, and set on both $\mathbb{A}_{G\times X}$ and  $\mathbb{A}_X$ the correspondent compact-open topology. Furthermore, denote the canonical projections by $\pi_X:G\times X\to X$ and $\pi_G:G\times X\to G$.
  
  \smallskip
  
  There exists a bijective correspondence between continuous cotranslations $Z:G\times G\to \mathbb{A}_X$ and continuous group morphisms $W:G\to \mathbb{A}_{G\times X}$ such that $ (\pi_G\circ W(g))(h,x)=gh$ for every $x\in X$. Moreover, the correspondence is given by
  $$Z(h,g)(x)=\left(\pi_x\circ W(g)\right)(h,x),$$
  and 
  $$W(g)(h,x)=\left(gh,Z(h,g)(x)\right),\qquad \forall\, g,h,\in G,\,x\in X.$$
\end{proposition}

\begin{proof}
The continuity of $Z$ follows from $W$ being continuous and vice versa. Take $g_1,g_2,h\in G$ and $x\in X$. If $W$ is a group morphism, then
\begin{align*}
\left[Z(g_1h,g_2)\circ Z(h,g_1)\right](x)&=Z(g_1h,g_2)\left( \left[\pi_X\circ W(g_1)\right](h,x)\right)\\
    &=\left[\pi_X\circ W(g_2)\right]\left(g_1h,\left[\pi_X\circ W(g_1)\right](h,x)\right)\\
    &=\left[\pi_X\circ W(g_2) \circ W(g_1) \right](h,x)\\
    &=\left[\pi_X\circ W(g_2g_1) \right](h,x)\\
    &=Z(h,g_2g_1)(x),
\end{align*}
thus $Z$ is a cotranslation. On the other hand, if $Z$ is a cotranslation, then
\begin{align*}
    \left[W(g_2)\circ W(g_1)\right](h,x)&=W(g_2)\left(g_1h,Z(h,g_1)x\right)\\
    &=\left(g_2g_1h,\left[Z(g_1h,g_2)\circ Z(h,g_1)\right](x)\right)\\
    &=\left(g_2g_1h,Z(h,g_2g_1)(x)\right)\\
    &=W(g_2g_1)(h,x),
\end{align*}
thus $W$ is a group morphism.
\end{proof}

It is worth noting that a crucial condition in the previous proposition is that $G$ is itself a topological space, {\it i.e.} an object in the same category as $X$, so that the object $\mathbb{A}_{G\times X}$ has a clear meaning. Hence, in order to generalize this conclusion to other categories, we must in general  consider a group object of the desired category (we refer the reader to \cite[Chapter 4]{Awodey}). 

\smallskip

Nevertheless, this does not trivialize the construction that cotranslations present, since by changing the topological space in which the dynamics occurs, we change qualitative descriptions of said dynamics. For instance, in the case of a nonautonomous differential equation on $\mathbb{R}^d$, say $\dot{x}=f(t,x)$, the construction on Proposition \ref{202} coincides with the usual definition of an autonomous differential equation on $\mathbb{R}^{d+1}$ by taking $y=(t,x)$ and $\dot{y}=(1,f(y))$. However, it is well known that qualitative descriptions of the second equation are not well translated to the first.

\smallskip

Having established the main properties of cotranslations, we now turn our attention to the specific case when we consider the group $\mathbb{Z}$. This will shed light on the relevance of cotranslations in the context of nonautonomous difference equations.

\begin{proposition}%\label{176}
Set $X$ a Banach space, thus $\mathbb{B}_X$ is the Banach algebra of linear continuous operators and $\mathbb{A}_X$ is the topological group of its homeomorphic isomorphisms. There is a bijective correspondence between cotranslations $Z:\mathbb{Z}\times\mathbb{Z}\to \mathbb{A}_X$ and nonautonomous linear difference equations $x(n+1)=A(n)x(n)$, with $\mathbb{Z}\ni n\mapsto A(n)\in \mathbb{A}_X$, where the correspondence is given by
    \begin{equation*}
    Z(n,m)= \left\{ \begin{array}{lcc}
             A(n+m-1)A(n+m-2)\cdots A(n) &  \text{ if } &   m>0 \\
             \\ \Id & \text{ if }& m=0 \\
             \\ A^{-1}(n+m)A^{-1}(n+m+1)\cdots A^{-1}(n-1) & \text{ if } & m<0.
             \end{array}
   \right.
\end{equation*}
\end{proposition}
\begin{proof}
By defining $Z$ as in the statement of the proposition, starting from the function $n\mapsto A(n)\in \mathbb{A}_X$, clearly we obtain a groupoid morphism. Conversely, if we start with a cotranslation it is enough to define $A(n)=Z(n,1)$ and we obtain the desired equation.
\end{proof}

The next section will present a continuous analogue to the preceding result. Concluding this section, we provide a generalization of the previous proposition for finitely generated groups. We omit the proof since it follows the same steps as before. However, before delving into this generalization, we introduce an auxiliary definition.

\begin{definition}
    Let $G$ be a discrete group with $n$ generators $\{\xi_1,\dots,\xi_n\}$ verifying the set of relations $R$. For each word $p\in R$ there is a positive integer $|p|$ called \textbf{length of the word} and a map $j_p:\{1,\dots,|p|\}\to \{1,\dots,n\}$ such that 
    $$p=\prod_{k=1}^{|p|}\xi_{j_p(k)}=\xi_{j_p(|p|)}\cdot\xi_{j_p(|p|-1)}\,\cdots\,\xi_{j_p(2)}\cdot\xi_{j_p(1)}.$$

    Let $X$ be a Banach space. We say that a collection of maps $A_i:G\to\mathbb{A}_X$, $i=1,\dots,n$, \textbf{preserves relations} if for each $p\in R$ and $\eta\in G$ one has
   \begin{align*}
       \Id_X&=\prod_{k=1}^{|p|}A_{j_p(k)}\left(\left[\prod_{\ell=1}^{k-1}\xi_{j_p(\ell)}\right]\eta\right)\\
       &:=A_{j_p(|p|)}(\xi_{j_p(|p|-1)}\cdots\xi_{j_p(1)}\eta)\circ \cdots \circ A_{j_p(2)}(\xi_{j_p(1)}\eta)\,\circ\, A_{j_p(1)}(\eta).
   \end{align*} 
\end{definition}

\begin{proposition}\label{460}
    Let $G$ be a discrete group with $n$ generators $\{\xi_1,\dots,\xi_n\}$ and $X$ an object on some category $\mathscr{C}$. A multivariable nonautonomous difference equation of $G$ on $X$ is an equation of the form
    $$x(\xi_i\eta)=A_i(\eta)x(\eta),$$
    where the collection of maps $A_i:G\to \mathbb{A}_X$, $i=1,\dots,n$ preserves relations. A solution to this equation is a map $x:G\to X$. There is a bijective correspondence between multivariable nonautonomous difference equations of $G$ on $X$ and cotranslations $Z:G\rimes G\to \mathbb{A}_X$, which is given by $$A_i(\eta)=Z\left(\eta,\xi_i\right).$$
\end{proposition}

\subsection{Examples}
In the following we present various examples of cotranslations on diverse groups. These examples possess the characteristic of not corresponding to a classical dynamical system, yet they still encapsulate dynamics generated from these groups. Through the examples, the space $\mathbb{R}^d$ is provided its usual topology. On the other hand, for a collection of spaces $\{U_\lambda\}_{\lambda\in \Lambda}$, the space $\amalg_{\lambda\in \Lambda}\,U_\lambda$ is provided the disjoint union topology, while $\Pi_{\lambda\in \Lambda}\,U_\lambda$ has the product topology. 

\smallskip

A virtue of the groups we study here is that all of them allow for a canonical normal form, {\it i.e.} a canonical choice of writing of elements corresponding to their generator set. This normal form on the groups allows for an easier to verification that the maps we use indeed preserve the group relations. Finally, the group unit is always denoted by $e$. 

\smallskip

We begin by characterizing cotranslations of cyclic groups $C_n=\langle a \,| \,a^n=e\rangle$.  On any topological space $X$, take $f_0,\dots,f_{n-1}$ homeomorphisms of $X$ on itself verifying $f_{n-1}\circ \cdots\circ f_0=\Id_X$. Then, defining $A:C_n\to \mathbb{A}_X$ by $A(a^k)=f_{k}$, it clearly preserves the group relations, hence by Proposition \ref{460}, it defines a cotranslation. Note that the only way that this cotranslation corresponds to a group morphism is $f_0=\cdots=f_{n-1}$, while the notion of cotranslation allows for a more general choice of homeomorphisms. In particular, each pair of  $f_k$, $f_j$ does not need to commute.

\begin{example}
{\rm
    Consider $C_3=\langle a\,|\, a^3=e\rangle$, the space $X=\mathbb{R}$ and the functions $f(x)=x+1$ and $g(x)=2x$. Setting $A: C_3 \rightarrow \mathbb{A}_X$ given by $A(e)(x)=f(x)$, $A(a)(x)=g(x)$ and $A(a^2)(x)=(f^{-1}\circ g^{-1})(x)=\frac{x}{2} -1$, we can define the cotranslation $Z: C_3\times C_3 \rightarrow \mathbb{A}_X$ given by     
    \begin{table}[h]
    \centering
    \begin{tabular}{ccc}
        $Z(e,e)=\Id_X$, & $Z(a,e)= \Id_X,$& $Z(a^2,e)=\Id_X,$ \\
        $Z(e,a)(x)= x+1,$ & $Z(a,a)(x)= 2x,$ & $Z(a^2,a)(x)= \frac{x}{2}-1,$\\
        $Z(e,a^2)(x)= 2x+2,$ & $Z(a,a^2)(x)=x-1,$& $Z(a^2,a^2)(x)= \frac{x}{2}.$
    \end{tabular}
\end{table}

Remarkably, there are no non-trivial classic group action of $C_3$ on $\mathbb{R}$. Moreover, in this case clearly not all maps commute.
    }
\end{example}

\begin{comment}
\begin{example}
    {\rm \color{violet}
    Let $C_6=\langle\ a\ |\ a^6=1\ \rangle$ and $X=\mathbb{R}^3$.
    Given a real number $\gamma$, let us denote by $\theta_\gamma$ the counterclockwise rotation in the angle $\gamma$ of plane $XY$ and by $\sigma_\gamma$ the the counterclockwise rotation in the angle $\gamma$ of plane $YZ$.
    As $\theta_\gamma$ and $\sigma_\gamma$ denote rotations in $\mathbb{R}^3$, it is clear that $\theta_\gamma\circ\theta_\zeta=\theta_{\gamma+\zeta}$, $\sigma_\gamma\circ\sigma_\zeta=\sigma_{\gamma+\zeta}$ and $\theta_\gamma\circ\sigma_\zeta=\sigma_\zeta\circ \theta_\gamma$ for every $\gamma,\zeta\in\mathbb{R}$.

    Let us now consider two trios of real numbers $\{\alpha_0,\alpha_1,\alpha_2\}$ and $\{\beta_0,\beta_1,\beta_2\}$ such that $\alpha_0+\alpha_1+\alpha_2$ and $\beta_0+\beta_1+\beta_2$ are a integer multiple of $2\pi$. Define $A:C_6 \rightarrow \mathbb{A}_X$ by
    \begin{table}[h]
    \centering
    \begin{tabular}{ccc}
       $A(1)=\theta_{\alpha_0}$, & $A(a)=\theta_{\alpha_1}$, & $A(a^2)=\theta_{\alpha_2}$,\\
       $A(a^3)=\sigma_{\beta_0}$,  & $A(a^4)=\sigma_{\beta_1}$, & $A(a^5)=\sigma_{\beta_2}$.
    \end{tabular}
\end{table}

    As we have $\theta_{\alpha_0}\circ\theta_{\alpha_1}\circ\theta_{\alpha_2}=\sigma_{\beta_0}\circ\sigma_{\beta_1}\circ\sigma_{\beta_2}=id_X$, the function $A$ respects the group structure of $C_6$. 
    Therefore, $A$ defines a cotranslation of $Z: C_6 \times C_6 \rightarrow \mathbb{A}_X$.
    }
\end{example}
\end{comment}

\begin{example}

    {\rm 
    Consider $C_{n}=\langle\, a\, |\, a^n=e\, \rangle$, with $n\geq 2$, and $X=\mathbb{R}^{d}$.
    Take $k$ non-empty finite sets of different real numbers $R_1$, $R_2$, \dots $R_k$, such that $\Sigma_{r \in R_i}\, r\in 2\pi \mathbb{Z}$ for every $i\in\{1,2,\dots, k\}$ and $\Sigma_{i=1}^k\ m_i=n$, where  $m_i=|R_i|$.

    \smallskip

Choose $\pi_1,\dots,\pi_k$ planes on $\mathbb{R}^d$ containing the origin and denote by $\theta_i(r)$ the counterclockwise rotation of $\mathbb{R}^d$ by the angle $r$ respect to a fixed normal vector to the plane $\pi_i$, for $r$ in $R_i$ and $i\in\{1,2,\dots, k\}$.  Set $\Theta=\left\{\theta_i(r):r\in R_i,\, i\in\{1,\dots,k\}\right\}$ and fix a bijection $A:C_n\rightarrow \Theta\subset \mathbb{A}_X$. As every rotation $\theta_i(r)$ commutes with each $\theta_j(r')$, we have
    $$A(a^{n-1+t})\circ \dots \circ A(a^{1+t})\circ A(a^t)=\Id_X,$$
    for every $0\leq t\leq n-1$, hence $A$ defines a cotranslation. We name this dynamics as \textit{multi-rotational cotranslations}.
    }
\end{example}

\begin{comment}
\begin{example}
    {\rm  \color{violet}
    Let $C_6=\langle\ a\ |\ a^6=1\ \rangle$ be the cyclic group of six elements and let $X=\mathbb{R}^2\times \{-1\}\amalg\mathbb{R}^2\times \{1\}$. Consider two functions $A,B: C_3 = \langle\ r\ |\ r^3=1\ \rangle \rightarrow \mathbb{A}_{\mathbb{R}^2}$ such that  $A(r^{i+2})\circ A(r^{i+1}) \circ A(r^i) = B(r^{i+2})\circ B(r^{i+1})\circ B(r^i)= \Id_{\mathbb{R}^2}$ for $i=0,1,2.$

    Given a real number $x$, denote by $\phi_x$ the canonical homeomorphism from $\mathbb{R}^2$ to $\mathbb{R}^2\times \{x\}$. We define $\tilde{A}:C_3 \rightarrow \mathbb{A}_{\mathbb{R}^2\times{1}}$ and $\tilde{B}:C_3 \rightarrow \mathbb{A}_{\mathbb{R}^2\times{-1}}$ by $\tilde{A}(r^i):=\phi_{1}\circ A(r^i)\circ \phi_{1}^{-1}$ and $\tilde{B}(r^i):=\phi_{-1}\circ B(r^i)\circ \phi_{-1}^{-1}$ for $i=0,1,2$. Now define $C:C_6 \rightarrow \mathbb{A}_x$ by

    \begin{table}[h]
    \centering
    \begin{tabular}{ccc}
        $C(1)=\tilde{A}(1)$ & $C(a)=\tilde{A}(r)$ & $C(a^2)=\tilde{A}(r^2) $, \\
        $C(a^3)=\tilde{B}(1)$ & $C(a^4)=\tilde{B}(r)$ & $C(a^5)=\tilde{B}(r^2). $
    \end{tabular}
    \end{table}

    As $X$ is a disjoint union of $\mathbb{R}^2\times \{-1\}$ and $\mathbb{R}^2\times \{1\}$, we have $\tilde{A}(r^i)\circ \tilde{B}(r^j) = \tilde{B}(r^j)\circ \tilde{A}(r^i)$ for all $i,j=0,1,2.$ Therefore, $C$ preserves the relations of $C_6$ and it defines a cotranslation of $C_6$ on $X$.
    }
\end{example}
\end{comment}

\begin{example}
    {\rm 
    Fix a positive integer $n$ and consider $C_{3n}=\langle\, a\, |\, a^{3n}=e\, \rangle$ and $X=\amalg_{i=1}^n\ S^1$. For each $i\in\{1,\dots, n\}$, set a map $A^i:C_3=\langle\, r\, |\, r^{3}=e\, \rangle\rightarrow \mathbb{A}_{S_i^1}$, where $S_i^1$ corresponds to the $i$-th copy of $S^1$ in $X$, such that it defines a cotranslation of $C_3$ on $S_i^1$. 

\smallskip
    
    Note that each $A^i:C_3\rightarrow \mathbb{A}_{S_i^1}$ defines a map $\tilde{A}^i: C_3 \rightarrow \mathbb{A}_X$ through $\left[\tilde{A}^i(r^j)\right](x)=\left[A^i(r^j)\right](x)$ if $x\in S_i^1$ and $\left[\tilde{A}^i(r^j)\right](x)=x$ if not. With these functions we can define $A:C_{3n}\rightarrow \mathbb{A}_X$ given by $A(a^{3i + j})=\tilde{A}^{i+1}(r^j)$ for all $i\in \{0,\dots,n+1\}$ and $j\in\{0,1,2\}$. 

\smallskip
    
    As each $\tilde{A}^i(r^\ell)$ acts non-trivially only on the $i$-th copy of $S^1$, we have $\tilde{A}^i(r^\ell)\circ\tilde{A}^j(r^k) =\tilde{A}^j(r^k)\circ \tilde{A}^i(r^\ell) $ for all $i, j\in\{1,\dots, n\}$ and $\ell,k\in\{0,1,2\}$. Therefore, $A$ defines a cotranslation of $C_{3n}$ on $X$. We name this dynamics as {\it disjoint cotranslations}. Note that under a similar construction, we can define a disjoint cotranslation of $C_{n_1+\dots+n_p}$ by leveraging cotranslations of $C_{n_1},\dots, C_{n_p}$. 
    }
\end{example}

Until now, every example corresponds to a 1-generated group, consequently, according to Proposition \ref{460}, we only needed to consider a single function $A:G\to \mathbb{A}_X$. In the following we study examples with groups generated by two o more elements.

\smallskip

Consider for instance the dihedral group of $2n$ elements $D_{2n}=\langle r,s\, |\, r^n=s^2=(sr)^2=e \rangle =\{e, r, \dots, r^{n-1}, s, sr, \dots, sr^{n-1}\}$. Let $X$ be some object in a category. In order to define a cotranslation of $D_{2n}$ on $X$, it is enough to consider two functions $A_r,A_s: D_{2n} \rightarrow \mathbb{A}_X$ satisfying:
    \begin{align*}
    A_r(r^{i+n-1})\circ \dots \circ A_r(r^{i+1})\circ A_r(r^i) &= \Id_X, \\
    A_r(sr^{i-n+1})\circ \dots \circ A_r(sr^{i-1})\circ A_r(sr^i) &= \Id_X, \\
    A_s(sr^i) \circ A_s(r^i) = A_s(r^i) \circ A_s(sr^i) &= \Id_X, \\
    A_s(sr^i)\circ A_r(sr^{i+1})\circ A_s(r^{i+1}) \circ A_r(r^i) &= \Id_X, \\
    A_s(r^i)\circ A_r(r^{i-1})\circ A_s(sr^{i-1}) \circ A_r(sr^i) &= \Id_X,
\end{align*}
for all $i\in \{0,1,\dots, n-1\}$.

\begin{comment}
\begin{example}
    {\rm \color{violet}
In particular, consider $n=3$ and $X=\mathbb{R}^2\times \{-1\}\amalg\mathbb{R}^2\times \{1\}$. Moreover, let $\alpha,$ $\beta$ and $\gamma$ three real numbers such that $\alpha+\beta+\gamma$ is a integer multiple of $2\pi$.

For this example, given a real number $\theta$, let us denote by $R_{-1,\theta}$ and $R_{1,\theta}$ the automorphisms of $X$ such that 

$$R_{\epsilon,\theta}(v,\delta) = \left \{ 
\begin{array}{ccc}
    (R_{\theta}(v),\delta) & if & \epsilon=\delta,\\
    (v, \delta) & if & \epsilon \neq \delta,
\end{array}
\right .$$
for $\epsilon$ in $\{-1,1\}$.

Furthermore, let us denote by $\sigma$ the automorphism of $X$ such that $\sigma(v,\delta) = (v, (-1)\cdot \delta)$ for all $v$ in $\mathbb{R}^2$ and $\delta$ in $\{-1,1\}$.

Finally, we can define the functions $A_r,A_s:D_6 \rightarrow \mathbb{A}_X$, where

\begin{table}[h]
    \centering
    \begin{tabular}{ccc}
        $A_r(1)= R_{1,\alpha}$ & $A_r(r)= R_{1,\beta}$ & $A_r(r^2)= R_{1,\gamma}$ \\
        $A_r(s)= R_{-1,-\gamma}$ & $A_r(sr)= R_{-1,-\beta}$ & $A_r(sr^2)= R_{-1,-\alpha}$ \\
    \end{tabular}
\end{table}
and $A_s(g)=\sigma$ for all $g$ in $D_6$. It's easy to verify that the functions $A_r$ and $A_s$ satisfies the previous conditions, therefore these define a cotranslation $Z$ of $D_6$ on $X= \mathbb{R}^2\times \{-1\}\amalg\mathbb{R}^2\times \{1\}$.
    }
\end{example}
\end{comment}

\begin{example}
    {\rm 
    Take the group $D_6=\langle r,s \,|\, r^3=s^2=(sr)^2=e \rangle$ and three real numbers $\alpha,$ $\beta$ and $\gamma$ such that $\alpha+\beta+\gamma\in 2\pi \mathbb{Z}$. For a real number $\zeta$, the counterclockwise rotation of $\mathbb{R}^2$ by the angle $\zeta$ is denoted by $\theta_\zeta$. Let us consider $X= \amalg_{k\in \mathbb{Z}}\ \mathbb{R}^2\times \{k\}$. Given a real number $\zeta$, set $R_{e,\zeta}$ and $R_{o,\zeta}$ the homeomorphisms of $X$ given by
   $$R_{e,\zeta}(v,k)= \left \{
   \begin{array}{ccc}
      (\theta_{\zeta}(v), k)  &\text{if} &  k \text{ is even,}\\
      (v , k)  & \text{if} & k \text{ is odd,}
   \end{array}
   \right .$$
and
   $$R_{o,\zeta}(v,k)= \left \{
   \begin{array}{ccc}
      (v,k)  & \text{if} & k \text{ is even,} \\
      (\theta_{\zeta}(v), k) & \text{if}  & k \text{ is odd.}
   \end{array}
   \right .$$
Furthermore, let us denote by $\sigma$ the homeomorphism of $X$ given by $\sigma(v, k)=(v,k-1)$ for every $v\in\mathbb{R}^2$ and $k\in\mathbb{Z}$. Now, we define the functions $A_r,A_s:D_6 \rightarrow \mathbb{A}_X$ given by 
   \begin{table}[h]
   \centering
    \begin{tabular}{ccc}
        $A_r(e)=R_{e,\alpha}$ & $A_r(r)=R_{e,\beta}$ & $A_r(r^2)=R_{e,\gamma}$, \\
        $A_r(s)=R_{o,-\gamma}$ & $A_r(sr)=R_{o,-\beta}$ & $A_r(sr^2)=R_{o,-\alpha},$ 
    \end{tabular}
\end{table}\\
as well as $A_s(e)=A_s(r)=A_s(r^2)=\sigma$ and $A_s(s)=A_s(sr)=A_s(sr^2)=\sigma^{-1}$. It is easy to check that $A_r$ and $A_s$ satisfy the conditions above, therefore define a cotranslation of $D_6$ on $X$. Note that this cotranslation corresponds to a group action only if $\alpha=\beta=\gamma$.
    }
\end{example}

\begin{example}
    {\rm 
 More generally, suppose that a map $A:C_n\to \mathbb{A}_X$ defines a cotranslation of the cyclic group $C_n=\langle\, a\, |\, a^n=e\, \rangle$ on a topological space $X$.  Then, take the space $Y= \amalg_{k\in \mathbb{Z}}\, X$, with elements denoted by $y=(y_k)_{k\in \mathbb{Z}}$ and once again consider $\sigma$ the shift on $Y$ given by $\sigma(y)_{k}= y_{k-1}$ for $k\in\mathbb{Z}$ and $y\in Y$. Define the functions $A_e,A_o:C_n\rightarrow \mathbb{A}_Y$ given, for $g\in C_n$, by
    $$(A_e(g)(y))_{(k)}= \left \{ 
    \begin{array}{ccc}
        A(g)(y_k) & \text{if} & k \text{ is even}, \\
         y_k & \text{if} & k \text{ is odd},
    \end{array}
    \right .$$
and 
    $$(A_o(g)(y))_{(k)}= \left \{ 
    \begin{array}{ccc}
        y_k & \text{if} & k \text{ is even}, \\
         A(g)(y_k) & \text{if} & k \text{ is odd}.
    \end{array}
    \right .$$
    Finally, considering $D_{2n}=\langle r,s \,|\, r^n=s^2=(sr)^2=e \rangle$, define the maps $B_r,B_s: D_{2n} \rightarrow \mathbb{A}_Y$ by
    \[B_r(r^i)=A_e(a^i),\quad B_r(sr^i) = (A_o(a^{n-1-i}))^{-1},\quad B_s(r^i)=\sigma\quad \text{and}\quad B_s(sr^i)=\sigma^{-1},\]
    for $i\in\{0,1,\dots , n-1\}$. It can be verified that these maps preserve the relations of $D_{2n}$, therefore generating a cotranslation of $D_{2n}$ on $Y$.
    }    
\end{example}

Let us now study the infinite dihedral group $D_{\infty}=\langle\, a,b\, |\, a^2=b^2=e\, \rangle$. According to Proposition \ref{460}, to define a cotranslation on a space $X$ we need two functions $A_a,A_b: D_{\infty} \rightarrow \mathbb{A}_X$ verifying $A_a(aw)\circ A_a(w)=\Id_X$ and $A_b(bw)\circ A_b(w)= \Id_X$ for every $w\in D_{\infty}$.

\smallskip

For this, choose $\{g_n\}_{n\geq 0}$ and $\{h_n\}_{n\geq 0}$ two arbitrary sequences of homeomorphisms in $\mathbb{A}_X$. Then, we can define the functions $A_a$ and $A_b$ by $A_a(au)=g_{|u|}$ and $A_a(u)=g_{|u|}^{-1}$, where $u$ is a word not beginning with $a$,  and $A_b(bv)=h_{|v|}$ and $A_b(v)=h_{|v|}^{-1}$, where $v$ is a word not beginning with $b$. The dynamics emerging from this construction describe a classic dynamical system only if the sequences $\{g_n\}_{n\geq 0}$ and $\{h_n\}_{n\geq 0}$ are constant and each one is a fixed homeomorphism which is its own inverse.

\begin{comment}
\begin{example}
{\rm \color{violet}
    In particular, if we consider $X=\mathbb{R}^{2}$ and the automorphisms ${f_1}_{(x,y)}=(x+1,y)$, ${f_2}_{(x,y)}=(x,y+1)$, we can define the sequences $\{g_n\}_{n\geq 0}$, $\{h_n\}_{n\geq 0}$ such that $g_n=f_1 \circ \dots \circ f_1$ ($n$ times) and $h_n=f_2 \circ \dots \circ f_2$ ($n$ times) for all $n\geq 1$ and $g_0=h_0=\Id_X$. If we define the function $A$ as previously, we can generate a cotranslation $Z: D_{\infty}\times D_{\infty}\rightarrow \mathbb{A}_X$ of $D_{\infty}$ on $X$ in therms of Proposition \ref{460}. The evaluation of some elements of $D_{\infty}$ are presented below:

    $$ Z(bab,ab)_{(x,y)} = (A_a(ab)\circ A_b(bab))_{(x,y)} = (g_{|b|} \circ h_{|ab|})_{(x,y)} = (f_1 \circ f_2^2)_{(x,y)} = (x+1, y+2),$$      
    $$Z(abab, ab)_{(x,y)} = (A_a(babab)\circ A_b(abab))_{(x,y)} = (g_{|babab|}^{-1}\circ h_{|abab|}^{-1})_{(x,y)} = (f_1^{-5}\circ f_2^{-4})_{(x,y)} = (x-5, y-4), $$
    \begin{align*}
        Z(ab, ababa)_{(x,y)} &= (A_a(ba) \circ A_b(a) \circ A_a(1) \circ A_b(b) \circ A_a(ab))_{(x,y)}\\
        &=(g_{|ba|}^{-1}\circ h_{|a|}^{-1}\circ g_{|1|}^{-1}\circ h_{|1|} \circ g_{|b|})_{(x,y)}\\
        &= (f_1^{-2}\circ f_2^{-1}\circ id_X \circ id_X \circ f_1)_{(x,y)}\\
        &= (x-1,y-1).
    \end{align*}
}
\end{example}    
\end{comment}

\begin{example}
    {\rm 
    Consider the infinite dihedral group $D_{\infty}$ and $X=\mathbb{R}^2$. Consider the homeomorphism $f(x,y)=(x+1, y)$ and the real numbers sequence $\{\zeta_n= \frac{\pi}{2^{n+1}}\}_{n \geq 0}$. Again, let us denote by $\theta_\zeta$ the counterclockwise rotation of $\mathbb{R}^2$ by the angle $\zeta$.

    \smallskip

    We define the sequences $\{g_n\}_{n\geq 0}$, $\{h_n\}_{n\geq 0}$ by $g_n= f^n$  and $h_n=\theta_{\zeta_n}$ for all $n\geq 0$. If we define the maps $A_a$ and $A_b$ as above, we generate a cotranslation $Z: D_{\infty}\times D_{\infty}\rightarrow \mathbb{A}_X$ as in Proposition \ref{460}. Explicitly, and among others, we have

    $$Z(bab,ab)=g_{|b|}\circ h_{|ab|}= f\circ \theta_{\zeta_2},$$
    $$Z(abab,ab) = g_{|babab|}^{-1} \circ h_{|abab|}^{-1} = f^{-5} \circ \theta_{-\zeta_4},$$
    $$Z(ab,ababa) = g_{|ba|}^{-1}\circ h_{|a|}^{-1}\circ g_{|1|}^{-1}\circ h_{|1|} \circ g_{|b|} = f^{-2} \circ \theta_{-\zeta_1} \circ \Id_X \circ \theta_{\zeta_0} \circ f = f^{-2} \circ \theta_{\zeta_0 -\zeta_1}\circ f.$$
    }
\end{example}

\begin{example}
    {\rm 
    Let $X=\mathbb{R}^2$, consider a sequence $\{L_n\}_{n\geq 0}$ of straight lines going through the origin and a sequence $\{\zeta_n\}_{n\geq 0}$ of real numbers such that $0<\zeta_n\leq\pi$ for all $n\geq 0$. In order to construct a cotranslation, as previously, we define the sequences $\{g_n\}_{n\geq 0}$ and $\{h_n\}_{n\geq 0}$ of homeomorphisms, where $g_n$ denotes the symmetry of $\mathbb{R}^2$ respect to $L_n$, while $h_n$ denotes the rotation $\theta_{\zeta_n}$, with the same notation as before, for all $n\geq 0$. Then, the functions $A_a,A_b : D_{\infty} \rightarrow \mathbb{A}_X$ generate a cotranslation of $D_{\infty}$ on $X=\mathbb{R}^2$ as above. Note that the only case in which this dynamics are given by a classic action is if all the lines are the same and $\zeta_n=\pi$ for every $n\in \mathbb{N}$. 
    }
\end{example}

\begin{comment}
\begin{example}\label{2002} 
    {\rm \textcolor{red}{quiza podemos prescindir de este y quedarnos solo con el 3.23}
    Take the set $\textbf{2}=\{0,1\}$ provided the discrete topology and consider the topological space $X=\textbf{2}^{\mathbb{N}_0}=\Pi_{n\in \mathbb{N}_0}\textbf{2}$. We denote the elements of $X$ by $x=(x_k)_k$.
    
\smallskip

    Define the function $\overline{(\ )}: \{0,1\}\rightarrow \{0,1\}$ by $\overline{0}=1$ and $\overline{1}=0$. Now, consider $\sigma:X \rightarrow X$ given by $\sigma(x)_0 = \overline{x_0}$ and $\sigma(x)_k =x_k$ for every $k\geq 1$. We know $\sigma$ describes an homeomorphism on $X$. 

\smallskip

   \textcolor{red}{hay que revisar esta parte aún} In addition, let us denote the permutation $(0\ n)$ in $\mathbb{N}_0$ by $\tau_n$ , for all $n\geq 1$. From the previous notation, we define, for all $n\geq1$, the functions $\rho_n: X\rightarrow X$ where $\rho_n(x)_k = x_{\tau_n(k)}$ for all $k\geq 0$. Each $\rho_n$ is also a homeomorphism of $X$ because of the previous argument.

\smallskip

    Finally, in order to define a cotranslation of $D_{\infty}$ on $X$, we can define the sequences $\{g_n\}_{n\geq 0}$ and $\{h_n\}_{n\geq 0}$ such that $g_n=\sigma$ for all $n\geq0$, $h_n=\rho_n$ for all $n\geq 1$ and $h_n=\Id_X$ \textcolor{red}{esto tengo la impresión de que tiene un error, quizá el segundo era $h_0$}. Then, defining the functions $A_a,A_b$ as previously, we generate a cotranslation of $D_{\infty}$ on $X$.
    }
\end{example}
\end{comment}

A group with a similar structure to the infinite dihedral is $C_2\ast C_3 = \langle\, a,b\, |\, a^2=b^3=e\, \rangle$, the free product of $C_2$ and $C_3$. On this group every non-trivial element $w$ can be written uniquely in one of the forms below: 
\[ab^{j}\cdots ab^{k}a,\quad ab^{j}\cdots ab^{k},\quad b^{j}\cdots ab^{k}\quad\text{or}\quad b^{j}\cdots ab^{k}a,\]
where $j,k\in\{0,1,2\}$.  Therefore, every non-trivial element $w$ satisfies either $w=au$ for a unique $u$ not beginning with $a$, or $w=b^{j}v$ for $j\in\{1,2\}$ and a unique element $v$ not beginning with a power of $b$.

\smallskip
    
    According to Proposition \ref{460}, a cotranslation of $C_2\ast C_3$ on $X$ is defined by two functions $A_a,A_b:C_2\ast C_3 \rightarrow \mathbb{A}_X$ such that $A_a(aw)\circ A_a(w)=\Id_X$ and $A_b(b^2w)\circ A_b(bw)\circ A_b(w)=\Id_X$ for all $w$ in $C_2\ast C_3$. 
    Thus, in order to define a cotranslation, we only need a sequence $\{g_n\}_{n\geq 0}$ in $\mathbb{A}_X$ and a sequence $\{h_{0,n},\ h_{1,n},\ h_{2,n}\}_{n\geq0}$ of automorphisms set such that $h_{2,n}\circ h_{1,n} \circ h_{0,n}=\Id_X$ for every $n\geq 0$. 

\smallskip

    Hence, we can define functions $A_a$ and $A_b$ given by $A_a(au)=g_{|u|}$, $A_a(u)=g_{|u|}^{-1}$ for every $u$ not beginning with $a$ and $A_b(b^2v)=h_{2,|v|}$, $A_b(bv)=h_{1,|v|}$, $A_b(v)=h_{0,|v|}$ for all elements $v$ which do not begin with a power of $b$.

\begin{example}
{\rm  
%Let us consider \textbf{3}$=\{0,1,2\}$ and $X=\textbf{3}^{\mathbb{N}_0}$ provided the prodiscrete topology. 
Take the set \textbf{3}$:=\{0,1,2\}$ and consider the space $X=\textbf{3}^{\mathbb{N}_0}= \Pi_{n\in \mathbb{N}_0} \textbf{3}$ provided the prodiscrete topology. We denote the elements of $X$ by $x=(x_k)_k$. On $\textbf{3}$, we denote the permutations $(0\ 1)$, $(1\ 2)$ and $(0\ 1\ 2)$ by $\sigma_0$, $\sigma_1$ and $\sigma_2$, respectively.  
    Let us define homeomorphisms $\Sigma_0, \Sigma_1, \Sigma_2: X \rightarrow X$ given by $\Sigma_j(x)_0=\sigma_j(x_0)$ and $\Sigma_j(x)_k = x_k$ for every $k\geq 1$ and $j\in\{0,1,2\}$.    

\smallskip

In addition, let us denote the permutation $(0\ n)$ in $\mathbb{N}_0$ by $\tau_n$, for all $n\geq 1$. Furthermore, we define, for all $n\geq1$, the functions $\rho_n: X\rightarrow X$ given by $\rho_n(x)_k = x_{\tau_n(k)}$ for all $k\geq 0$. Each $\rho_n$ is also a homeomorphism of $X$. 

\smallskip

Then, we can consider functions $A_a,A_b: C_2\ast C_3 \rightarrow \mathbb{A}_X$ given by $A_a(au)=\tau_{|u|}$, $A_a(u)=\tau_{|u|}^{-1}$ for every $u$ not beginning with $a$ and $A_b(b^jv)=\Sigma_j$ for $v$ not beginning with a power of $b$, for $j\in\{0,1,2\}$. Then, the maps $A_a$ and $A_b$ define a cotranslation of $C_2\ast C_3$ on $X$.

\smallskip

Note that we can analogously define a cotranslation of the group $C_2\ast C_n$ on the space $\textbf{n}^{\mathbb{N}_0}:=\{0,\dots,n-1\}^{\mathbb{N}_0}$.
    }
\end{example}

Other family of finitely generated groups with great relevance on dynamics are free groups. Consider for instance the free group of rank 2, denoted by $F_2$ with generators $a$ and $b$. As this group admits a presentation without non-trivial relations, then any couple of functions $A_a,A_b:F_2\to \mathbb{A}_X$ can define a cotranslation, according to Proposition \ref{460}.

\smallskip

In the following examples we explore dynamics of $F_2$ on the category of rooted tree graphs. 

\smallskip

Let $L$ be a set of $d$ elements. We denote by $L^*$ the set of all words of finite length on alphabet $L$. We regard $L^*$ as a rooted tree graph, as in \cite[Chapter 1]{Nekra}. A function $f:L^*\rightarrow L^*$ is an endomorphism if it preserves adjacency of the vertices. If $f$ is also a bijective function, it is called automorphism.

\begin{comment}
In the following examples we explore dynamics of $F_2$ on the category of rooted tree graphs \cite{Nekra}. Let $L$ be a set of $d$ elements, with $d\geq 2$. We denote by $L^n$ the set of all words of length $n$ on alphabet $L$ and $L^*:=\bigcup_{n=0}^{\infty}\ L^n$ the set of all the finite length words on the alphabet $L$. We denote the empty word by $e$.

\smallskip

Now, we define the graph $T^{(d)}$ whose vertex are $L^*$ and two vertices $w$ and $v$ are connected if and only if there exists $a$ in $L$ such that $v=wa$. In particular, $T^{(d)}$ is a tree whose root is $e$. From now on, we referrer to $T^{(d)}$ as the $d$-ary tree.

\smallskip

A function $f:L^*\rightarrow L^*$ is an endomorphism if it preserves adjacency of the vertices, {\it i.e.} if $v$ and $w$ satisfy $v=wa$ for some $a$ in $L$, then there exists some $b$ in $L$ such that $f(v)=f(w)b$.  Moreover, a map $f:T^{(d)}\rightarrow T^{(d)}$ is called automorphism if it is a bijective endomorphism \textcolor{red}{quiza esto se puede reducir un poco agregando una referencia}.\textcolor{blue}{ De hecho, todo aparece 'textualmente' en [23]} A direct consequence of this definition is that an automorphisms of $T^{(d)}$ maps words of length $n$ to words of length $n$, for all $n\geq 1$. In particular, it fixes the root $e$.

\smallskip
\end{comment}

Given a word $w$ in $L^*$, we define the subtree $wL^*:=\{ wv\ |\ v\in L^* \}$. It admits a canonical injective morphism $\phi_w:L^*\rightarrow wL^*$, given by $\phi_w(v)=wv$ for every $v\in L^*$. Having said that and considering an automorphism $g$ of $L^*$, we define the automorphism $\tilde{g}_{(w)}$ such that 
\begin{equation*}
    \tilde{g}_{(w)}(v)=\left \{ \begin{array}{ccc}
       v  & \text{if} & v \not\in wL^*, \\
        wg(v') & \text{if} & v=wv'\in wL^*.
    \end{array} \right .
\end{equation*}

\begin{example}
    {\rm 
    Let us consider $\textbf{2}:=\{0,1\}$, the binary tree $\textbf{2}^*$ and the graph morphism $\phi: \{a^{\pm},b^{\pm}\}^{*}\rightarrow \textbf{2}^{*}$ which is induced  by $\phi(a)=\phi(a^{-1})=1$ and $\phi(b)=\phi(b^{-1})=0$. Let us also consider $\sigma$ in $\mathbb{A}_{\textbf{2}^*}$ defined by $\sigma(0v)=1v$ and $\sigma(1v)=0v$ for all $v$ in $\textbf{2}^*$.

\smallskip

Consider the free group $F_2=\langle \,a,b\,|\ \rangle$. Then, we define maps $A_a,A_b: F_2 \rightarrow \mathbb{A}_{\textbf{2}^*}$ by $A_a(w)=A_b(w)=\tilde{\sigma}_{(\phi(w))}$ for every $w\in F_2$. As $F_2$ does not admit any non-trivial relation, these functions define a cotranslation of $F_2$ on the binary tree, as desired.
    
    }
\end{example}

{\rm
The previous example has an clear disadvantage in comparison to others: the cotranslation defined is not injective. In particular, we can generate the same tree automorphism from different elements in the groupoid. 
}

\begin{example}
    {\rm  
    Consider again the group $F_2$ with the previous presentation and the tree $L^*$, where $L=\{a^{\pm}, b^{\pm}\}$. 
    Let $\sigma_a$ and $\sigma_b$ be a couple of permutations of $L$. Given these permutations, we can define the maps $A_a,A_b: F_2 \rightarrow \mathbb{A}_{L^*}$ given by $A_a(w)=\tilde{\sigma_a}_{(w)}$ and $A_b(w)=\tilde{\sigma_b}_{(w)}$ for all $w\in F_2$. 
    Once again,  these functions allows us define a cotranslation of $F_2$ on the tree $L^*$. Besides, as every element in $F_2$ is in correspondence with a unique vertex in $L^*$, the functions $A_a$ and $A_b$ are injective. 
    }
\end{example}

\begin{example}
    {\rm 
    More generally, let us consider the free group of rank $n$, which we denote by $F_n$ with free generators $a_1$, $a_2$, \dots $a_n$, and the $2n$-ary tree $L^*$, where $L=\{a_1^{\pm}, a_2^{\pm}, \dots, a_n^{\pm}\}$. 
    Given a finite sequence $\sigma_1$, $\sigma_2$, \dots, $\sigma_n$ of permutations of $L$, we can define the functions $A_i: F_n \rightarrow \mathbb{A}_{L^*}$ given by $A_i(w)=\tilde{\sigma_i}_{(w)}$ for all $w\in F_n$ and $i\in\{1,2,\dots, n\}$.

\smallskip

    Finally, the functions $\{A_i\}_{i=1}^n$ define a cotranslation of $F_n$ on the $2n$-ary tree $L^*$.
    }
\end{example}

It should be emphasized that every time that a group $G$ acts on the binary tree through automorphisms, it induces a unique continuous action of $G$ on the topological Cantor space $\textbf{2}^{\mathbb{N}} $ \cite[p. 2]{Nekra}. In the same fashion, all the examples provided here for cotranslations of free groups can be translated from the category of rooted trees to the category of topological spaces.

\smallskip

In the reminder of this section we delve into the problem of defining new cotranslations from others previously known. For this purpose, we leverage the concept of free product of groups \cite[Chapter 2]{Harpe}. Particularly, we highlight that given two groups, with presentation $G= \langle\, S\, |\, R\, \rangle$ and $H=\langle\, T\, |\, Q\, \rangle$, their free product admits the presentation $G\ast H= \langle\, S\cup T\, |\, R\cup Q\, \rangle$. \cite[Chapter 4]{Lyndon}. For instance, we have $D_\infty=C_2\ast C_2$ and $F_2=\mathbb{Z}\ast \mathbb{Z}$.

\smallskip

Furthermore, the free product $G\ast H$ supports two canonical injective maps $\iota_G: G\rightarrow G\ast H$ and $\iota_H:H\rightarrow G\ast H$, and from its universal property \cite{Harpe}, we have that given a group $K$ and two morphisms $\phi_G: G \rightarrow K$, $\phi_H: H \rightarrow K$, there is a unique morphism $\phi: G\ast H \rightarrow K$ such that $\phi_G= \phi \circ \iota_G$ and $\phi_H= \phi \circ \iota_H$.

\smallskip

In particular, if we consider $K=G$, we denote by $\pi_G: G\ast H \rightarrow G$ the unique morphism such that $\Id_G= \pi_G \circ \iota_G$ and $\rho_{G,H} = \pi_G \circ \iota_H$, where $\rho_{G,H}$ denotes the trivial morphism from $G$ to $H$. The map $\pi_H$ is defined analogously when $K=H$.

\smallskip

Let us now establish a result regarding cotranslations of free product groups.

\begin{proposition}
    Let $G= \langle\, S\, |\, R\, \rangle$, $H=\langle\, T\, |\, Q\, \rangle$ be finitely presented groups and $X$ a topological space. According to Proposition \ref{460}, let us suppose the functions $\{A_s:G \rightarrow \mathbb{A}_X\}_{s\in S}$ and $\{A_t:H \rightarrow \mathbb{A}_X\}_{t\in T}$ define cotranslations of $G$ and $H$ respectively on $X$. Now, for every $s\in S$ and $t\in T$, let us define $\tilde{A}_s: G\ast H \rightarrow \mathbb{A}_X$, $\tilde{A}_t: G\ast H \rightarrow \mathbb{A}_X$ by $\tilde{A}_s(w)=A_s(\pi_G(w))$ and $\tilde{A}_t(w)=A_t(\pi_H(w))$ for every element $w$. 

    \smallskip
    
    Then, the functions $\{\tilde{A}_s\}_{s \in S}$ and $\{\tilde{A}_t\}_{t\in T}$ define a cotranslation of $G\ast H$ on $X$.
\end{proposition}

\begin{proof}
     It is enough to verify that functions $\{\tilde{A}_s\}_{s \in S}$ and $\{ \tilde{A}_t\}_{t\in T}$ preserve the relations of $G\ast H$.

   \smallskip

    Let $r=s_n\dots s_2s_1$ be a relation in $R$. By hypothesis, $$A_{s_n}(s_{n-1}\dots s_1g)\circ \dots \circ A_{s_2}(s_1g)\circ A_{s_1}(g) = \Id_X,$$ 
    for every $g\in G$. Now, consider an arbitrary $w\in G\ast H$. Since $\pi_G(s)=s$ for every $s\in S$, we have
    \begin{align*}
    \tilde{A}_{s_n}(s_{n-1}\dots s_1w)\circ \dots \circ \tilde{A}_{s_1}(w) &= A_{s_n}(\pi_G(s_{n-1}\dots s_1w))\circ \dots \circ A_{s_1}(\pi_G(w))\\
    &= A_{s_n}(s_{n-1}\dots s_1\pi_G(w))\circ \dots\circ A_{s_1}(\pi_G(w))\\
    &=  \Id_X,
\end{align*}
    because $\pi_G(w)$ is an element of $G$. Since $r$ is an arbitrary relation in $R$, we conclude that functions $\{\tilde{A}_s\}_{s \in S}$ and $\{\tilde{A}_t\}_{t\in T}$ preserve every relation in $R$. Likewise, we can conclude that $\{\tilde{A}_s\}_{s \in S}$ and $\{\tilde{A}_t\}_{t\in T}$ respect every relation in $Q$. %Finally, the maps $\{\tilde{A}_s\}_{s \in S}$ and $\{\tilde{A}_t\}_{t\in T}$ define a cotranslation of $G\ast H$ on $X$.
\end{proof}

To conclude this section, let us develop one last observation.

\smallskip

Let $G$ be a group provided by the presentation $\langle\, S\, |\, R\, \rangle$. Suppose that there is a cotranslation $Z$ of $G$ on some space $X$, defined by functions $\{ A_s:G\rightarrow \mathbb{A}_X \}_{s\in S}$. If we consider a non-void subset $R_0$ of $R$, it is possible to define the group $H=\langle\, S\, |\, R_0\, \rangle$. There exist a canonical morphism $\pi: H\rightarrow G$ such that $\pi(s)=s$ for each generator $s$.  It is clear that $\pi$ is a surjective map, therefore $G$ is isomorphic to $H/\ker(\pi)$. Moreover, we can define the map $\tilde{A}_s:= A_s\circ \pi$ for all $s$ in $S$. 

\smallskip

If we consider a relation $r=s_n\dots s_2s_1$ in $R_0$, then we have
$$A_{s_n}(s_{n-1}\dots s_1g)\circ \dots \circ A_{s_2}(s_1g) \circ A_{s_1}(g) = \Id_X,$$
for all $g$ in $G$, because the family $\{A_s\}_{s\in S}$ preserves the relations $R$. Thus, for an arbitrary element $h$ in $H$, we have
\begin{align*}
    \tilde{A}_{s_n}(s_{n-1}\dots s_1h)\circ \dots \circ \tilde{A}_{s_2}(s_1h) \circ \tilde{A}_{s_1}(h) &= A_{s_n}(\pi(s_{n-1}\dots s_1h))\circ \dots \circ A_{s_2}(\pi(s_1h)) \circ A_{s_1}(\pi(h)) \\
    &= A_{s_n}(s_{n-1}\dots s_1\pi(h))\circ \dots \circ A_{s_2}(s_1\pi(h)) \circ A_{s_1}(\pi(h)) \\
    &= \Id_X,
\end{align*}
due to $\pi(h)$ being in $G$. As $r$ is an arbitrary relation in $R_0$, we conclude that the functions $\{ \tilde{A}_s:H \rightarrow \mathbb{A}_X \}_{s\in S}$ preserve the relations of $H$. In conclusion, this defines a cotranslation of $H$ on the space $X$.

\smallskip

From this observation, we conclude the following particular claim: If we consider the groups $D_6=\langle\, r,s\, |\, r^3=s^2=(sr)^2=e\, \rangle$, $C_3\ast C_2 = \langle\, r,s\, |\, r^3=s^2=e\, \rangle$ and $F_2 = \langle\, r,s\, |\, \rangle$, we clearly have $\emptyset \subset \{r^3, s^2\}\subset \{r^3, s^2, (sr)^2\}$. Therefore,  all cotranslations of $D_6$ on some space $X$ are inherited in this manner to $C_3\ast C_2$, in the same way that all cotranslation of $C_3\ast C_2$ on $X$ are inherited to $F_2$.

\section{Differentiable cotranslations on Banach spaces}
In this section we study the differentiability of cotranslations and analyze their relation to differential equations. During this section we fix a Banach space $X$ over the field $\mathbb{K}$, which can be $\mathbb{R}$ or $\mathbb{C}$. We set
\begin{itemize}
\item $\mathbb{L}_X$ the collection of all densely defined linear operators on $X$,
    \item $\mathbb{B}_X$ the algebra of all continuous elements of $\mathbb{L}_X$, given the strong topology,
    \item $\mathbb{A}_X$ the group of all invertible elements of $\mathbb{B}_X$ whose inverse is also in $\mathbb{B}_X$. %as a topological subspace of $\mathbb{B}_X$.
\end{itemize}

The group associated with the cotranslations studied in this section is the field $\mathbb{K}$. Nevertheless, the formalism presented here is versatile, enabling generalizations of the results to other non-commutative Lie groups. We use Banach spaces so that $\mathbb{B}_X$ has the algebra structure, thus we are able to study derivatives for groupoid morphisms through limits. 

\smallskip

    Consider $\varphi:\mathbb{K}\rimes\mathbb{K}\to \mathbb{B}_X$ and $\psi:\mathbb{K}\to \mathbb{B}_X$. In this section we use the following notation (independently if these limits exist or not):
    \begin{itemize}
        \item [i)] $\partial_1\varphi (r,t)=\lim_{h\to 0}\frac{\varphi(r+h,t)-\varphi(r,t)}{h}$,
        \item [ii)] $\partial_2\varphi (r,t)=\lim_{h\to 0}\frac{\varphi(r,t+h)-\varphi(r,t)}{h}$,
        \item [iii)] $\frac{d}{du}\left[\psi(u)\right]=\lim_{h\to 0}\frac{\psi(u+h)-\psi(u)}{h}$.
    \end{itemize}

\smallskip

Notation iii) proves useful when applying derivatives to functions obtained as compositions of other functions or when emphasizing the {\it variable} with respect to which we intend to differentiate. On the other hand, notations i) and ii) are designed to underscore the {\it position} of the variable with respect to which we seek differentiation.We begin by stating basic results .

\begin{lemma}\label{153}
    Let $Z:\mathbb{K}\rimes \mathbb{K}\to \mathbb{A}_X$ be a continuous cotranslation. If for every $t\in \mathbb{K}$ the map $r\mapsto Z(r,t)$ is derivable at $r=r_0$, for some $r_0\in \mathbb{K}$, then the map $r\mapsto Z^{\inv}(r,t)$ is derivable at $r=r_0$ and
    $$\partial_1Z^\inv(r_0,t)=-Z^{\inv}(r_0,t)\Big[\partial_1Z(r_0,t)\Big] Z^{\inv}(r_0,t).$$
\end{lemma}
\begin{proof}
    By the following
    \begin{eqnarray*}
        Z(r,t) Z^\inv(r,t)=\Id&\Rightarrow&\frac{d}{dr}\left[Z(r,t)Z^\inv(r,t)\right]=0\\
        &\Leftrightarrow&\Big[\partial_1Z(r,t)\Big]\ Z^\inv(r,t)+Z(r,t) \left[\partial_1Z^\inv(r,t)\right]=0\\
        &\Leftrightarrow&\partial_1 Z^\inv(r,t)=-Z^\inv(r,t) \Big[\partial_1 Z(r,t)\Big] Z^\inv(r,t),
    \end{eqnarray*}
    it is easy to see that the derivatives $\partial_1$ at $r=r_0$ exist simultaneously for both $r\mapsto Z(r,t)$ and $r\mapsto Z^\inv(r,t)$.
\end{proof}

The following result is obtained with an analogous proof.

\begin{lemma}\label{171}
       Let $Z:\mathbb{K}\rimes \mathbb{K}\to \mathbb{A}_X$ be a continuous cotranslation. If for every $r\in \mathbb{K}$ the map $t\mapsto Z(r,t)$ is derivable at $t_0$, for some $t_0\in \mathbb{K}$, then the map $t\mapsto Z^\inv(r,t)$ is derivable at $t=t_0$ and
       $$\partial_2Z^\inv(r,t_0)=-Z^{\inv}(r,t_0)\Big[\partial_2Z(r,t_0)\Big] Z^{\inv}(r,t_0).$$
\end{lemma}

\begin{lemma}%\label{169}
    Let $Z:\mathbb{K}\rimes \mathbb{K}\to \mathbb{A}_X$ be a continuous cotranslation. Suppose that for every $t\in\mathbb{K}$ the map $r\mapsto Z(r,t)$ is derivable at $r=0$. Then, for every $t\in \mathbb{K}$ the function $r\mapsto Z(r,t)$ is derivable and
    $$\partial_1 Z(r,t)=\Big[\partial_1 Z(0,t+r)\Big] Z(r,-r)-Z(r,t)\Big[\partial_1Z(0,r)\Big] Z(r,-r).$$
\end{lemma}

\begin{proof}
We suppose the following limit exists
$$\lim_{h\to 0}\frac{Z(0+h,t)-Z(0,t)}{h}.$$

We have
\begin{eqnarray*}
    \frac{Z(r+h,t)-Z(r,t)}{h}&=&\frac{Z(h,t+r)Z^\inv(h,r)-Z(0,r+t)Z^\inv(0,r)}{h}\\
    \\&=&\frac{Z(h,t+r)-Z(0,r+t)}{h}Z^\inv(h,r)\\
   \\&&+Z(0,r+t)\frac{Z^\inv(h,r)-Z^\inv(0,r)}{h}.
\end{eqnarray*}

Both addends at the right hand side have limit when $h\to 0$ by hypothesis (and Lemma \ref{153}), thus the left hand side as limit as well, hence $r\mapsto Z(r,t)$ is derivable at every point for every fixed $t\in\mathbb{K}$ and
\begin{eqnarray*}
    \partial_1Z(r,t)&=&\Big[\partial_1Z(0,t+r)\Big]Z^\inv(0,r)+Z(0,r+t)\Big[\partial_1Z^{\inv}(0,r)\Big]\\
    &=&\Big[\partial_1Z(0,t+r)\Big]Z(r,-r)-Z(0,r+t)Z(r,-r)\Big[\partial_1Z(0,r)\Big]Z(r,-r)\\
    &=&\Big[\partial_1Z(0,t+r)\Big]Z(r,-r)-Z(r,t)\Big[\partial_1Z(0,r)\Big]Z(r,-r).
\end{eqnarray*}
\end{proof}

Analogously we have the following result.

\begin{lemma}\label{165}
    Let $Z:\mathbb{K}\rimes\mathbb{K}\to \mathbb{A}_X$ be a continuous cotranslation. Suppose that for every $r\in\mathbb{K}$ the function $t\mapsto Z(r,t)$ is derivable at $t=0$. Then, for every $r\in \mathbb{K}$ the function $t\mapsto Z(r,t)$ is derivable and
    $$\partial_2Z(r,t)=\Big[\partial_2Z(r+t,0)\Big]Z(r,t).$$
\end{lemma}

\begin{proof}
    Suppose that for every $r\in \mathbb{K}$ the following limit exists
    $$\lim_{h\to 0}\frac{Z(r,h)-Z(r,0)}{h}.$$

    We have
    \begin{eqnarray*}
        \frac{Z(r,t+h)-Z(r,t)}{h}&=&\frac{Z(r+t,h)Z(r,t)-Z(r,t)}{h}\\
        &=&\frac{Z(r+t,h)-\Id}{h}Z(r,t)\\
        &=&\frac{Z(r+t,h)-Z(r+t,0)}{h}Z(r,t).
    \end{eqnarray*}

By taking limits $h\to 0$ we obtain the desired identity.
\end{proof}

Now, we present a result demonstrating how to deduce derivability with respect to one coordinate when we have information about the other.

\begin{lemma}\label{147}
    Let $Z:\mathbb{K}\rimes\mathbb{K}\to \mathbb{A}_X$ be a continuous cotranslation such that for every $r\in \mathbb{K}$ the map $t\mapsto Z(r,t)$ is derivable. Then, the map $r\mapsto Z(r,t)$ is derivable for every $t\in \mathbb{K}$ and
    $$\partial_1Z(r,t)=\partial_2Z(r,t)-Z(r,t) \Big[\partial_2Z(r,0)\Big].$$
\end{lemma}

\begin{proof}
    Suppose that for every $r,t\in \mathbb{K}$ the following limit exists
    $$\lim_{h\to 0}\frac{Z(r,t+h)-Z(r,t)}{h}.$$
    
    Note that
     \begin{eqnarray*}
        \frac{Z(r,t+h)-Z(r,t)}{h}&=&\frac{Z(r+h,t)Z(r,h)-Z(r,t)}{h} \\
        \\&=&\frac{Z(r+h,t)Z(r,h)-Z(r,t)Z(r,h)}{h}+\frac{Z(r,t)Z(r,h)-Z(r,t)}{h} \\
        \\&=&\frac{Z(r+h,t)-Z(r,t)}{h}Z(r,h)+Z(r,t)\frac{Z(r,h)-Z(r,0)}{h}.
    \end{eqnarray*}
    from where, reorganizing terms and taking limits we obtain
    \begin{eqnarray*}
        \lim_{h\to 0}\frac{Z(r+h,t)-Z(r,t)}{h}&=&\Big[\partial_2Z(r,t)\Big] Z(r,0)-Z(r,t)\Big[ \partial_2Z(r,0)\Big]\\
        \\&=&\partial_2Z(r,t)-Z(r,t)\Big[\partial_2Z(r,0)\Big].
    \end{eqnarray*}
\end{proof}

As a summary of the previous lemmas we state the following:

\begin{corollary}
    For a continuous cotranslation $Z: \mathbb{K}\rimes\mathbb{K}\to \mathbb{A}_X$, the following statements are equivalent:
    \begin{itemize}
        \item [i)] $t\mapsto Z(r,t)$ is derivable at $t=0$, for every $r\in \mathbb{K}$.
        \item [ii)] $t\mapsto Z(r,t)$ is derivable for every $r\in \mathbb{K}$.
    \end{itemize}

    Moreover, each one of them implies the following statements, which are equivalent:
    \begin{itemize}
        \item [iii)] $r\mapsto Z(r,t)$ is derivable at $r=0$, for every $t\in \mathbb{K}$.
        \item [iv)] $r\mapsto Z(r,t)$ is derivable for every $t\in \mathbb{K}$.
    \end{itemize}

\end{corollary}

    From now on, we say that a cotranslation is \textbf{differentiable} if it verifies i) or ii) on the previous corollary. In the following we study the relation of differentiable cotranslations and linear nonautonomous differential equations. Even in the autonomous case, the existence of solutions is a non trivial problem when $X$ is infinite dimensional. An alternative is to look for solutions with a restricted domain as  $[0,\infty)$ or $[0,t]$ (we refer the reader to  \cite[Chapter 4]{Pazy}). 

\smallskip
On the other hand, the nonautonomous case presents even greater difficulties. On finite dimension, it is known that it is enough to ask $t\mapsto A(t)$ to be locally integrable in order to guarantee the existence and uniqueness of solutions. On infinite dimensional spaces, the problem is much harder. A partial result stays that if $t\mapsto A(t)\in \mathbb{B}_X$ is continuous under the uniform operator norm, then we have the existence and uniqueness of solutions defined on a bounded interval \cite[Theorem 5.1.1]{Pazy}.

\smallskip

In general, the problem of the existence and uniqueness of globally defined solutions is not fully resolved. In the concluding part of this section we address this issue by leveraging the structure of cotranslations.

\begin{proposition}\label{178}
Set $Z:\mathbb{K}\rimes\mathbb{K}\to \mathbb{A}_X$ a differentiable cotranslation. Define $A:\mathbb{K}\to \mathbb{L}_X$ and the operator $\Psi:\mathbb{K}^2\to \mathbb{A}_X$ by
$$A(u):=\partial_2Z(u,0),\qquad\Psi(u,v):=Z(v,u-v),$$ 
    then, the following are verified:
    \begin{itemize}
    \item [i)] $\Psi(u,v)\Psi(v,w)=\Psi(u,w)$ for every $u,v,w\in \mathbb{K}$,
        \item [ii)] $\frac{d\Psi}{du}=A(u)\Psi(u,v)$,
        \item [iii)] $\frac{d\Psi}{dv}=-\Psi(u,v)A(v)$,
        \item [iv)] for every $\xi\in X$, the map $\psi_{v,\xi}:\mathbb{K}\to X$, given by $\psi_{v,\xi}(u)=\left[\Psi(u,v)\right](\xi)$ verifies $\psi_{v,\xi}(v)=\xi$ and is a solution to the equation 
    \begin{equation}\label{170}
    \frac{dx}{du}=A(u)x(u).
    \end{equation}
    \end{itemize}
\end{proposition}

\begin{proof}
It is easy to see that each $A(u)$ is a linear transformation of $X$ (although we cannot ensure in general that it is continuous). Similarly, we cannot in general state that $u\mapsto A(u)$ is continuous. To verify i) it is enough to see:
$$\Psi(u,v)\Psi(v,w)=Z(v,u-v)Z(w,v-w)=Z\left(w,(v-w)+(u-v)\right)=Z(w,u-w)=\Psi(u,w).$$

On the other hand, on Lemma \ref{165} we proved the identity 
$$\partial_2Z(r,t)=\Big[\partial_2Z(r+t,0)\Big]Z(r,t),$$
hence
\begin{equation*}
 \frac{d\Psi}{du}(u,v)=\frac{d}{du}\left[Z(v,u-v)\right]=\partial_2Z(v,u-v)=\Big[\partial_2Z(v+u-v,0)\Big]Z(v,u-v)=A(u)\Psi(u,v),
\end{equation*}
from where ii) follows. Then, trivially $\psi_{v,\xi}$ is a solution to (\ref{170}) and
$$\psi_{v,\xi}(v)=\left[\Psi(v,v)\right](\xi)=\left[Z(v,v-v)\right](\xi)=\Id(\xi)=\xi,$$ 
thus verifying iv). Finally, note that
$$\Psi(u,v)=Z(v,u-v)=Z(0,u)Z(v,-v)=Z(0,u)Z^\inv(0,v),$$
hence, using the identity from Lemma \ref{171} we obtain
\begin{align*}
    \frac{d\Psi}{dv}(u,v)=Z(0,u)\frac{d}{dv}\left[Z^\inv(0,v)\right]&=Z(0,u)\Big[\partial_2Z^\inv(0,v)\Big]\\
    &=-Z(0,u)Z^\inv(0,v)\Big[\partial_2Z(0,v)\Big]Z^\inv(0,v)\\
    &=-\Psi(u,v)\Big[\partial_2Z(v,0)\Big]\\
    &=-\Psi(u,v)A(v),
\end{align*}
where the second to last equality follows from the identify of Lemma \ref{165}, thus proving iii).
\end{proof}
 
The existence of a function with the properties of $\Psi$ as described in the preceding is precisely what is required to articulate globally defined solutions for every initial condition. We define this concept as follows:

\begin{definition}
 \cite[Definition 5.1.3]{Pazy}   A map $\Psi:\mathbb{K}\rimes\mathbb
    {K}\to \mathbb{B}_X$ is an \textbf{evolution operator} (or evolution system) if the following conditions are verified:
    \begin{itemize}
        \item [i)] $\Psi(r,r)=\Id$, $\Psi(r,t)\Psi(t,s)=\Psi(r,s)$ for every $r,s,t\in \mathbb{K}$,
        \item [ii)] $(r,t)\to \Psi(r,t)\in\mathbb{B}_X$ is strongly continuous.
    \end{itemize}

    If furthermore there is a map $t\mapsto A(t)\in\mathbb{L}_X$ such that
    \begin{itemize}
        \item [iii)] $\partial_1\Psi(r,t)=A(r)\Psi(r,t)$ and $\partial_2\Psi(r,t)=-\Psi(r,t)A(t)$, 
        we say that $\Psi$ is the \textbf{evolution operator associated} to the differential equation $\dot{x}=A(t)x(t)$.
    \end{itemize}
    
\end{definition}

We know \cite[Theorem 4.1.3]{Pazy} that an autonomous linear differential equation $\dot{x}=Ax$ on a Banach space has uniquely defined solutions on $[0,\infty)$ for every initial condition if and only if $A$ is the infinitesimal generator \cite[Definition 1.2, p.49]{Engel} of a $C_0$-semigroup. Inspired by this, we present the following definition and theorem.

\begin{definition}
    For a differentiable cotranslation $Z:\mathbb{K}\rimes\mathbb{K}\to \mathbb{A}_X$, its \textbf{infinitesimal generator} is the function $A:\mathbb{K}\to \mathbb{L}_X$ given by
    $A(t)=\partial_2Z(t,0)$. 
\end{definition}

Note that the infinitesimal generator of a cotranslation corresponds to the derivative respect to the second coordinate evaluated on the unit space $\left(\mathbb{K}\rimes\mathbb{K}\right)^{(0)}$ of the groupoid.

\begin{theorem}
    A linear nonautonomous differential equation on a Banach space $X$
    \begin{equation}\label{179}
        \dot{x}(t)=A(t)x(t),
    \end{equation}
    has an evolution operator associated to it (i.e. unique and globally defined solution for every initial condition on $X$) if and only if $A$ is the infinitesimal generator of a differentiable cotranslation $Z:\mathbb{K}\rimes \mathbb{K}\to \mathbb{A}_X$.
\end{theorem}

\begin{proof}
    Proposition \ref{178} states that if $A(t)=\partial_2Z(t,0)$ for some cotranslation, then there is an evolution operator associated to (\ref{179}). 
    
    \smallskip
    On the other hand, if (\ref{179}) has an evolution operator associated, say $\Psi$, then set $Z:\mathbb{K}\rimes\mathbb{K}\to\mathbb{A}_X$  by $Z(r,t)=\Psi(t+r,r)$. It is easy to see that such $Z$ is a differentiable groupoid morphism and furthermore, by Lemma \ref{165}, we conclude $A(t)=\partial_2Z(t,0)$. 
\end{proof}

\section{Partial cotranslations}

In this section, we focus on a partial version of cotranslations. We fix a topological group $G$ and consider the left translations groupoid structure in $G \times G$. Additionally, we work with a Banach space $X$ over the field $\mathbb{K}$. While some results presented in this section apply to general Banach spaces, the key conclusions require $X$ to be finite-dimensional. Towards the end of the section, the notion of an orthonormal basis is entailed, so in essence, our main results are established when considering the space $X=\mathbb{K}^d$ with its Euclidean norm. In this case,  $\mathbb{B}_X\cong \mathcal{M}_d(\mathbb{K})$ and $\mathbb{A}_X\cong GL_d(\mathbb{K})$.

\begin{definition} 
    \begin{itemize}
        \item [a)] We say a function  $W:G\times G\to \mathbb{B}_X$ is a \textbf{partial cotranslation}  if
        $$W(g,kh)=W(hg,k)\circ  W(g,h),\quad\text{ for all } g,h,k\in G.$$

        \item [b)] For two partial cotranslations $W,V:G\rimes G\to\mathbb{B}_x$, we say they are \textbf{mutually orthogonal} if $W(hg,k)V(g,h)=V(hg,k)W(g,h)=0$ for all $g,h,k\in G$.
    \end{itemize}
\end{definition}

Unlike a cotranslation, a partial cotranslation is not necessarily a groupoid morphism, since its codomain is not a groupoid (at least not if we define $\mathbb{B}_X^{(2)}=\mathbb{B}_X^2$). Now we study some basic properties for cotranslations.

\begin{lemma}\label{331}
    If $W:G\rimes G\to \mathbb{B}_X$ is a partial cotranslation, then $W(g,e)$ is idempotent for all $g\in G$.
\end{lemma}

\begin{proof}
  By definition, we have $W(g,e)W(g,e)=W(g,e\cdot e)=W(g,e)$ for all $g\in G$.
\end{proof}

\begin{lemma}\label{352}
    Let $W:G\rimes G\to \mathbb{B}_X$ be a partial cotranslation. For all $g,h,k\in G$ one gets $\ker W(h,g)=\ker W(h,k)$. In particular, if $X$ is finite dimensional, then $\rank W(h,g)=\rank W(h,k)$.
\end{lemma}

\begin{proof}
    From the identity
    $$W(h,k)=W(gh,kg^{-1})W(h,g)$$
    one gets $\xi\in \ker W(h,g)\Rightarrow \xi\in \ker W(h,k)$, thus $\ker W(h,g)\leq \ker W(h,k)$ and by symmetry of the argument we obtain the equality. The second conclusion is trivial if $X$ is finite dimensional.
\end{proof}

A fundamental result in linear algebra that we require is that for two transformations $A$ and $B$ in $\mathbb{B}_X$, if $AB$ has the same kernel as $B$, then $\dim \ker A \leq \dim \ker B$.

\begin{proposition}\label{366}
      Let $W:G\rimes G\to \mathbb{B}_X$ be a partial cotranslation, where $X$ is finite dimensional. Then $\rank W(g,h)=\rank W(k,\ell)$ for every $g,h,k,\ell\in G$. 
\end{proposition}

\begin{proof}
    Starting from the identity
    $$W(hg,h^{-1})W(g,h)=W(g,e),$$
    as Lemma \ref{352} indicates $W(g,h)$ and $W(g,e)$ have the same kernel, then the previous observation implies $\dim\ker W(hg,h^{-1})\leq \dim\ker W(g,h)$. 

\smallskip
    Now, from the identity
    $$W(g,h)W(hg,h^{-1})=W(hg,e),$$
    as Lemma \ref{352} showed that $W(hg,h^{-1})$ and $W(hg,e)$ have the same kernel, then $\dim\ker W(g,h)\leq \dim \ker W(hg,h^{-1})$. In conclusion, $\dim\ker W(hg,h^{-1})=\dim\ker W(g,h)$, thus $\rank W(hg,h^{-1})=\rank W(g,h)$.

    \smallskip
    Taking $h=kg^{-1}$ in the last equality we obtain $\rank W(k,gk^{-1})=\rank W(g,kg^{-1})$, but for arbitrary $h,\ell\in G$ we know, from Lemma \ref{352} that $\rank W(g,kg^{-1})=\rank W(g,h)$ and $\rank W(k,gk^{-1})=\rank W(k,\ell)$, thus $\rank W(g,h)=\rank W(k,\ell)$.
\end{proof}

In view of the above, we can define without ambiguity:

\begin{definition}\label{365}
     Let $X$ be finite dimensional. For a partial cotranslation $W:G\rimes G\to\mathbb{B}_X$, the \textbf{rank} of the partial cotranslation is $\rank \,W:=\rank\, W(g,h)$, for arbitrary $g,h\in G$.
\end{definition}

As every cotranslation is in particular a partial cotranslation, the following definition applies as well for cotranslations.

\begin{definition}\label{158}
    Given a partial cotranslation $V:G\rimes G\to \mathbb{B}_X$, we say a function $\P:G\to \mathbb{B}_X$ is:
    \begin{itemize}
        \item [a)] \textbf{projector} if $\P(g)$ is idempotent on $\mathbb{B}_X$ for every $g\in G$.
        \item [b)] \textbf{invariant projector} associated to $V$ if it is a projector and
        $$\P(hg)V(g,h)=V(g,h)\P(g),\qquad\forall\,g,h\in G.$$
        \item [c)] For other invariant projectors $\mathrm{Q}:G\to \mathbb{B}_X$, we say $\P$ and $\mathrm{Q}$ are \textbf{mutually orthogonal} if $\P(g)\mathrm{Q}(g)=\mathrm{Q}(g)\P(g)=0$ for every $g\in G$.
    \end{itemize}
\end{definition}

It is easy to see that if $\P$ is an invariant projector for a partial cotranslation, then $\Id-\P$ is invariant as well, and they are mutually orthogonal. In the following we examine properties related to invariant projectors.

\begin{lemma}%\label{187}
    If $W:G\rimes G\to \mathbb{B}_X$ is a partial cotranslation and we define $\P:G\to \mathbb{B}_X$ by $\P(g)=W(g,e)$, called its \textbf{projector of the units space}, we obtain that $\P$ is an invariant projector associated to $W$. If $W$ is continuous, then its projector of the units space is continuous as well.
\end{lemma}

\begin{proof}
    We know from Lemma \ref{331} that it is indeed a projector. To obtain its invariance it is enough to see
    \begin{align*}
        W(g,h)\P(g)=W(g,h)W(g,e)=W(g,h)=W(hg,e)W(g,h)=\P(hg)W(g,h).
    \end{align*}

The last statement is trivial.
\end{proof}

\begin{proposition}\label{172}
    If $V:G\times G\to \mathbb{B}_X$ is a partial cotranslation and $\P$ is an associated invariant projector, then $W:G\rimes G\to \mathbb{B}_X$ given by $W(g,h)=V(g,h)\P(g)$ is a partial cotranslation.
    
    \smallskip
    If moreover both $\P$ and $V$ are continuous, then $W$ is continuous as well.
\end{proposition}

\begin{proof}
It is enough to see that
    \begin{eqnarray*}
        W(g,kh)=V(g,kh)\P(g)&=&V(g,kh)\P(g)^2\\
    &=&V(hg,k)V(g,h)\P(g)^2\\
    &=&V(hg,k)\P(hg)V(g,h)\P(g)\\
    &=&W(hg,k)W(g,h).
    \end{eqnarray*}

    Finally, continuity follows trivially.
\end{proof}

\begin{example}\label{417}
    {\rm Consider the linear nonautonomous  differential equation $\dot{x}=A(t)x$ with transition matrix $\Phi$. We know from Example \ref{191} that $Z(s,t)=\Phi(t+s,s)$ defines a continuous cotranslation with $X=\mathbb{R}^d$ and $G=\mathbb{R}$. Consider now a continuous invariant projector $\P$ for this cotranslation (in the sense of Definition \ref{158}). Proposition \ref{172} shows that $W(s,t)=Z(s,t)\P(s)$ defines a continuous partial cotranslation. 
    
    \smallskip
Moreover, in this case $\P$ is also an invariant projector for this differential equation, in the usual sense for nonautonomous dynamics \cite[p. 246]{Siegmund}. Indeed:
    \begin{equation*}
        \P(t)\Phi(t,s)=\P(t)Z(s,t-s)=\P\left((t-s)+s\right)Z(s,t-s)=Z(s,t-s)\P(s)=\Phi(t,s)\P(s).
    \end{equation*}
    
Note that $W$ can also be written $W(s,t)=\Phi(t+s,s)\P(s)$, which is a matrix function that describes the solutions to the equation $\dot{x}=A(t)x$ whose graph lies in the linear integral manifold \cite[Definition 2.2]{Siegmund} associated to the image of this projector \cite[Lemma 2.1]{Siegmund}.}
\end{example}

\begin{lemma}\label{372}
    If $W,V:G\rimes G\to \mathbb{B}_X$ are mutually orthogonal partial cotranslations, then $W+V:G\rimes G\to \mathbb{B}_X$ is a partial cotranslation. If both $W$ and $V$ are continuous, then $W+V$ is as well. 
\end{lemma}

\begin{proof}
Note that
        \begin{eqnarray*}
        W(g,kh)+V(g,kh)&=&W(hg,k)W(g,h)+V(hg,k)V(g,h)\\
    &=&W(hg,k)W(g,h)+W(hg,k)V(g,h)\\
    &&+V(hg,k)W(g,h)+V(hg,k)V(g,h)\\
    &=&\left(W(hg,k)+V(hg,k)\right)\left(W(g,h)+V(g,h)\right).
    \end{eqnarray*}

   Once again, continuity follows trivially. 
\end{proof}

\begin{lemma}\label{367}
    If $W,V:G\rimes G\to \mathbb{B}_X$ are mutually orthogonal partial cotranslations, then the unit space projector of $W$ is invariant for $W+V$. Moreover, the partial cotranslation obtained through $W+V$ and the projector of the units space of $W$ (following the construction on Proposition \ref{172}) is $W$.
\end{lemma}
\begin{proof}
   It is enough to note that 
   \begin{align*}
       W(hg,e)\left(W(g,h)+V(g,h)\right)&=W(hg,e)W(g,h)+W(hg,e)V(g,h)\\
       &=W(g,h)W(g,e)\\
       &=W(g,h)W(g,e)+V(g,h)W(g,e)\\
       &= \left(W(g,h)+V(g,h)\right)W(g,e).
   \end{align*}

For the second statement, note that $$W(g,h)=W(g,h)W(g,e)=\left(W(g,h)+V(g,h)\right)W(g,e).$$
\end{proof}

\begin{lemma}
    If $\P$ and $\mathrm{Q}$ are mutually orthogonal invariant projectors for a partial cotranslation $V:G\rimes G\to \mathbb{B}_X$, then the partial cotranslations  $V_\P,V_\mathrm{Q}:G\rimes G\to \mathbb{B}_X$, given by $V_\P(g,h)=V(g,h)\P(g)$ and $V_\mathrm{Q}(g,h)=V(g,h)\mathrm{Q}(g)$ are mutually orthogonal.
\end{lemma}

\begin{proof} We know both $V_\P$ and $V_\mathrm{Q}$ are partial cotranslations from Proposition \ref{172}. Moreover
    $$V_P(hg,k) V_\mathrm{Q}(g,h)=V(hg,k)\P(hg)V(g,h)\mathrm{Q}(g)=V(hg,k)V(g,h)\P(g)\mathrm{Q}(g)=0,$$
    and the other composition follows similarly.
\end{proof}

Now we present a mechanism to relate two different partial cotranslations.

\begin{definition}
    Given two partial cotranslations $W,V:G\times G\to \mathbb{B}_X$, we say they are \textbf{conjugated} if there is a map $T:G\to \mathbb{A}_X$ such that
$$T(hg)V(g,h)=W(g,h)T(g),\qquad\forall\,g,h\in G.$$
   
    In that case we call $T$ a \textbf{conjugation} between $W$ and $V$. If $T$, $W$ and $V$ are continuous, we say they are \textbf{continuously conjugated}. On the other hand, if 
    \begin{equation*}
        \sup_{g\in G}\left\{\norm{T(g)}, \norm{{T(g)^{-1}}}\right\}<\infty,
    \end{equation*}
   then we say $W$ and $V$ are \textbf{boudedly conjugated}
\end{definition}

The significance of conjugating two partial cotranslations lies in extracting information from one based on the other. Furthermore, as we progress through this section, we leverage this concept to demonstrate that every partial cotranslation can be represented as a cotranslation multiplied by an invariant projector.

\smallskip
In addition, the pursuit of continuous conjugations is motivated by the preservation of topological properties, while bounded conjugations are sought for their ability to preserve asymptotic behavior.

\begin{lemma}\label{369}
    Let $W:G\times G\to \mathbb{B}_X$ be a partial cotranslation and let $T:G\to \mathbb{A}_X$ be an un arbitrary map. Defining $W_T:G\rimes G\to \mathbb{B}_X$ by
    $$W_T(g,h)=T(hg)^{-1}W(g,h)T(g),$$
    we obtain that $W_T$ is a partial cotranslation. Moreover, if $X$ is finite dimensional, then $\rank W_T=\rank W$. Finally, if $T$ and $W$ are continuous, then $W_T$ is continuous as well.
\end{lemma}
\begin{proof}
    Note that
    \begin{align*}
        W_T(g,kh)&=T(khg)^{-1}W(g,kh)T(g)\\
        &=T(khg)^{-1}W(hg,k)W(g,h)T(g)\\
        &=T(khg)^{-1}W(hg,k)T(hg)T(hg)^{-1}W(g,h)T(g)\\
        &=W_T(hg,k)W_T(g,h).
    \end{align*}

    The second statement follows since $T(g)$ is invertible.
\end{proof}

\begin{lemma}\label{371}
    If $W,V:G\times G\to \mathbb{B}_X$ are two mutually orthogonal partial cotranslations and $T:G\to \mathbb{A}_X$ is an arbitrary map, then $W_T$ and $V_T$ (as defined on Lemma \ref{369}) are mutually orthogonal. 
\end{lemma}
\begin{proof}
It is enough to see that
\begin{align*}
    W_T(hg,k)V_T(g,h)&=T(khg)^{-1}W(hg,k)T(hg)T(hg)^{-1}V(g,h)T(g)^{-1}\\
    &=T(khg)^{-1}W(hg,k)V(g,h)T(g)^{-1}\\
    &=0,
\end{align*}
and the other composition follows similarly.
\end{proof}

From now on we fix $X=\mathbb{K}^d$ with the Euclidean norm. Thus, we replace $\mathbb{B}_X$ by $\mathcal{M}_d(\mathbb{K})$ and $\mathbb{A}_X$ by $GL_d(\mathbb{K})$.

\begin{proposition}\label{370}
 Every partial cotranslation $W:G\times G\to \mathcal{M}_d(\mathbb{K})$ is conjugated to a partial cotranslation $\widehat{W}:G\times G\to \mathcal{M}_d(\mathbb{K})$ whose projector of the units space is constant and orthogonal, i.e.
   $$\widehat{W}(g,e)=\begin{pmatrix}
\Id_{\rank W}&0\\
0&0
\end{pmatrix},\qquad\forall\,g\in G.$$
where $\Id_{\rank W}$ is the identity on $\mathbb{K}^{\rank W}$.
\end{proposition}

\begin{proof}
  For every $g\in G$ we know $W(g,e)$ is an idempotent (Lemma \ref{331}) of rank $\rank W$ (Proposition \ref{366}). Then, it follows that there exists some $T(g)\in GL_d(\mathbb{K})$ such that
    $$T(g)^{-1}W(g,e)T(g)=\begin{pmatrix}
\Id_{\rank W}&0\\
0&0
\end{pmatrix},$$
which defines a map $T:G\to GL_d(\mathbb{K})$. Defining $\widehat{W}:G\times G\to \mathbb{B}_X$ by $\widehat{W}=W_T$ as in Lemma \ref{369}, we obtain a partial cotranslation which is by construction conjugated to $W$ and verifies
$$\widehat{W}(g,e)=T(g)^{-1}W(g,e)T(g)=\begin{pmatrix}
\Id_{\rank W}&0\\
0&0
\end{pmatrix}.$$
\end{proof}

Note that the map $T$ we defined on the preceding proposition is not unique. We know eigenvectors (thus diagonalizations) of a continuous matrix function are not in general continuous or bounded. We dedicate the end of this section to deal with this fact.

\begin{proposition}\label{411}
Fix $\P:G\to \mathcal{M}_d(\mathbb{K})$ a projector of constant rank $n$ for which exists $M\geq 1$ such that 
    $$\sup_{g\in G}\left\{\norm{\P(g)},\norm{\Id-\P(g)}\right\}<M,$$
    then, there exists a map $T:G\to GL_d(\mathbb{K})$ such that
    \begin{equation}\label{410}
        T(g)^{-1}\P(g)T(g)=\begin{pmatrix}
\Id_n&0\\
0&0
\end{pmatrix},\qquad\forall\,g\in G,
    \end{equation}
  and
      $$\sup_{g\in G}\left\{\norm{T(g)},\norm{T(g)^{-1}}\right\}<\infty.$$
\end{proposition}

\begin{proof}
    Set $\xi_1,\dots,\xi_d$ the canonical basis of $\mathbb{K}^d$. For each $g\in G$, there is an orthonormal basis of $\im \,\P(g)$ given by $\{\eta_1^g,\dots,\eta_n^g\}$ and an orthonormal basis of $\ker\P(g)$ given by $\{\eta_{n+1}^g,\dots,\eta_d^g\}$. If we define a linear transformation $T(g):\mathbb{K}^d\to \mathbb{K}^d$ by
    $$T(g)\xi_i=\eta_i^g,$$
    evidently it verifies (\ref{410}).  Let $\zeta\in\mathbb{K}^d$ with $\norm{\zeta}=1$. There are unique $\alpha_1^\zeta,\dots,\alpha_d^\zeta\in \mathbb{K}$ with $|\alpha_i^\zeta|\leq 1$ such that
    $$\zeta=\sum_{i=1}^d\alpha_i^\zeta\xi_i,$$
then
\begin{align*}
    \norm{T(g)\zeta}=\norm{T(g)\left(\sum_{i=1}^d\alpha_i^\zeta\xi_i\right)}
    \leq\sum_{i=1}^d|\alpha_i^\zeta|\norm{T(g)\xi_i}
    \leq\sum_{i=1}^d\norm{\eta_i^g}=d,
\end{align*}
hence $\norm{T(g)}\leq d$. On the other hand, there exists unique $\beta_1^{\zeta,g},\dots,\beta_d^{\zeta,g}\in \mathbb{K}$ such that
$$\zeta=\sum_{i=1}^d\beta_i^{\zeta,g}\eta_i^g,$$
hence
$$\P(g)\zeta=\sum_{i=1}^n\beta_i^{\zeta,g}\eta_i^g,\qquad [\Id-\P(g)]\zeta=\sum_{i=n+1}^d\beta_i^{\zeta,g}\eta_i^g.$$

As $\norm{\P(g)}\leq M$, then $\norm{\P(g)\zeta}\leq M$, and as $\{\eta_1^g,\dots,\eta_n^g\}$ is an orthonormal basis, then $|\beta_i^{\zeta,g}|\leq M$ for every $i=1,\dots,n$. Analogously, as $\norm{\Id-\P(g)}\leq M$ and $\{\eta_{n+1}^g,\dots,\eta_d^g\}$ is an orthonormal basis, we obtain  $|\beta_i^{\zeta,g}|\leq M$ for every $i=n+1,\dots,d$. Thus 

\begin{align*}
    \norm{T(g)^{-1}\zeta}=\norm{T(g)^{-1}\left(\sum_{i=1}^d\beta_i^{\zeta,g}\eta_i^g\right)} &\leq\sum_{i=1}^d|\beta_i^{\zeta,g}|\norm{T(g)^{-1}\eta_i^g}\\
    &\leq\sum_{i=1}^dM\cdot\norm{\xi_i}=dM,
\end{align*}
hence $\norm{T(g)^{-1}}\leq dM$.
\end{proof}

The next corollary follows trivially from Propositions \ref{370} and \ref{411}.

\begin{corollary}\label{415}
    If $W:G\times G\to \mathcal{M}_d(\mathbb{K})$ is a partial cotranslation whose projector of the units space and its compliment are uniformly bounded, i.e. there exist $M\geq 1$ such that 
    \begin{equation}\label{414}
        \sup_{g\in G}\left\{\norm{W(g,e)},\norm{\Id-W(g,e)}\right\}<M,
    \end{equation}
then $W$ is boundedly conjugated to a partial cotranslation whose units space projector is constant and orthogonal.
\end{corollary}

To continue, we present a conjecture:

\begin{conjecture}\label{412}
   Denote the canonical basis of $\mathbb{K}^d$ by $\{\xi_1,\dots,\xi_d\}$. If $\P:G\to \mathbb{B}_X$ is a continuous projector, then for every $g\in G$ there exists an orthonormal basis of $\im \,\P(g)$ given by $\{\eta_1^g,\dots,\eta_n^g\}$ and an orthonormal basis of $\ker\P(g)$ given by $\{\eta_{n+1}^g,\dots,\eta_d^g\}$ such that defining $T:G\to GL_d(\mathbb{K})$ by
    $$T(g)\xi_i=\eta_i^g,$$
    we obtain that $T$ is continuous.
\end{conjecture}

Once again a corollary follows trivially from Proposition \ref{370} and the known fact that inversion of linear transformations is continuous. The second statement in the following is deduced trivially from Corollary \ref{415}.

\begin{corollary}\label{413}
     Suppose Conjecture \ref{412} is true. Then, every continuous partial cotranslation $W:G\times G\to \mathbb{B}_X$ is continuously conjugated to a partial cotranslation whose projector of the units space is constant and orthogonal.

     \smallskip
     Moreover, if the projector of the units space of $W$ and its compliment is uniformly bounded, {\it i.e.} it verifies (\ref{414}), then $W$ is continuously and boundedly conjugated to a partial cotranslation whose projector of the units space is constant and orthogonal.
\end{corollary}

\begin{remark}
    {\rm
    The preceding corollary, if Conjecture \ref{412} is true, presents a non-commutative and non-differentiable extension of S. Siegund's reducibility result for nonautonomous differential equations \cite[Theorem 3.2]{Siegmund2}. The key distinction lies in the fact that, in \cite{Siegmund2}, all invariant projectors are derived from a dichotomy.

\smallskip
Furthermore, it's noteworthy that when two partial cotranslations are boundedly and continuously conjugated, we are essentially describing a generalized notion of 'kinematic similarity' (see, for instance, \cite[Definition 2.1]{Siegmund2}), irrespective of whether the involved group is commutative or possesses a differential structure.

\smallskip
Remarkably, for discrete groups, the generalization holds trivially, as continuity is not a concern.
}
\end{remark}

To conclude, we state a theorem and consequent corollary to summarize the results of this section.

\begin{theorem}\label{368}
   Every partial cotranslation is completable to a cotranslation, i.e. for every partial cotranslation  $W:G\rimes G\to \mathcal{M}_d(\mathbb{K})$ there exists a partial cotranslation $V:G\times G\to \mathcal{M}_d(\mathbb{K})$ such that $W$ and $V$ are mutually orthogonal and $W+V$ has rank $d$. 

   \smallskip
  Moreover, if Conjecture \ref{412} is true and $W$ is continuous, then $V$ can also be chosen continuous.
\end{theorem}
\begin{proof}
Define the constant partial cotranslation $\widehat{V}:G\times G\to \mathcal{M}_d(\mathbb{K})$ by
$$\widehat{V}(g,h)=\begin{pmatrix}
0&0\\
0&\Id_{d-\rank W}
\end{pmatrix}.$$

We affirm $\widehat{V}$ is mutually orthogonal to $\widehat{W}$ (defined as in Proposition \ref{370}). Indeed:
\begin{align*}
\widehat{W}(hg,k)\widehat{V}(g,h)&=\widehat{W}(hg,k)\widehat{W}(hg,e)\widehat{V}(g,h)=\widehat{W}(hg,k)\begin{pmatrix}
\Id_{\rank W}&0\\
0&0
\end{pmatrix}\begin{pmatrix}
0&0\\
0&\Id_{d-\rank W}
\end{pmatrix}=0,
\end{align*}
and the other composition follows analogously. Choose $T$ as in Proposition \ref{370}. Denoting $T^{\inv}:G\to\mathbb{A}_X$ by $T^\inv(g)=T(g)^{-1}$, it follows by Lemma \ref{371} that $W=\widehat{W}_{T^\inv}$ and $V:=\widehat{V}_{T^{\inv}}$ are mutually orthogonal. Hence, by Lemma \ref{372}, $W+V$ is also a partial cotranslation. Moreover, for an arbitrary $g\in G$:
$$\rank \left(\widehat{W}+\widehat{V}\right)=\rank \left(\widehat{W}(g,e)+\widehat{V}(g,e)\right)=\rank\left(\begin{pmatrix}
\Id_{\rank W}&0\\
0&0
\end{pmatrix}+\begin{pmatrix}
0&0\\
0&\Id_{d-\rank W}
\end{pmatrix}\right)=d,$$
thus, by Lemma \ref{369} we have
$$\rank(W+V)=\rank \left(\widehat{W}+\widehat{V}\right)_{T^\inv}=\rank\left(\widehat{W}+\widehat{V}\right)=d.$$

The second statement follows trivially from Corollary \ref{413}.
\end{proof}

\begin{corollary}
    Every partial cotranslation is a cotranslation multiplied with an invariant projector.
\end{corollary}
\begin{proof}
The statement follows trivially from Theorem \ref{368} and Lemma \ref{367}.
\end{proof}

\section{Statements and Declarations}

There is no competing interest related to this project. This research has been partially supported by ANID, Beca de Doctorado Nacional 21220105

\end{document}